\documentclass[graybox]{svmult}

\usepackage{tkz-graph}
\usepackage{bbm}
\usepackage{mathtools}
\usepackage{young}
\usepackage{youngtab}
\usepackage{amsmath}
\usetikzlibrary{calc}
\usepackage{type1cm}        
\usepackage{makeidx}         
\usepackage{graphicx}        
\usepackage{multicol}        
\usepackage[bottom]{footmisc}

\usepackage{tikz}
\usepackage{tkz-berge}
\usepackage{newtxtext}       
\usepackage[varvw]{newtxmath}       

\makeindex             

\usepackage{ytableau}

\newcommand{\K}{\mathbb{K}}
\newcommand{\N}{\mathbb{N}}
\newcommand{\C}{\mathbb{C}}
\newcommand{\HH}{\mathbb{H}}
\newcommand{\Z}{\mathbb{Z}}
\newcommand{\A}{\mathcal{A}}
\newcommand{\R}{\mathbb{R}}
\newcommand{\1}{\mathbf{1}}
\newcommand{\B}{\mathcal{B}}
\renewcommand{\S}{\mathfrak{S}}
\newcommand{\G}{\Gamma}
\newcommand{\Van}{\Delta}
\newcommand{\abb}[5]{%
\setlength{\arraycolsep}{0.4ex}%
\begin{array}{rcccc}%
#1 &:\,& #2 & \,\,\longrightarrow\,\, & #3 \\[0.9ex]%
     & & #4 & \longmapsto & #5%
\end{array}%
}
\newcommand\restr[2]{{
  \left.\kern-\nulldelimiterspace 
  #1 
  \littletaller 
  \right|_{#2} 
  }}

\newcommand{\littletaller}{\mathchoice{\vphantom{\big|}}{}{}{}}
\DeclareMathOperator{\sgn}{sgn}

\DeclareMathOperator{\Tr}{Tr}
\DeclareMathOperator{\Sym}{Sym}
\DeclareMathOperator{\relint}{relint}
\DeclareMathOperator{\Stab}{Stab}
\DeclareMathOperator{\Hom}{Hom}
\DeclareMathOperator{\Com}{Com}
\DeclareMathOperator{\End}{End}
\DeclareMathOperator{\OO}{O}
\DeclareMathOperator{\ev}{eval}
\DeclareMathOperator{\CStab}{CStab}
\DeclareMathOperator{\RStab}{RStab}
\DeclareMathOperator{\Id}{Id}
\usepackage{ntheorem}

\ytableausetup{smalltableaux}
\begin{document}

\title*{Symmetries in polynomial optimization}

\author{Philippe Moustrou and Cordian Riener and Hugues Verdure}

\institute{Philippe Moustrou \at Institut de Mathématiques de Toulouse, Université Toulouse Jean Jaurès, 5 Allée Antonio Machado, 31058 Toulouse, France \email{philippe.moustrou@math.univ-toulouse.fr }
\and Cordian Riener \at Institutt for Matematikk og Statistikk, UiT - Norges arktiske universitet, 9037 Tromsø, Norway \email{cordian.riener@uit.no}
\and Hugues Verdure \at Institutt for Matematikk og Statistikk, UiT - Norges arktiske universitet, 9037 Tromsø, Norway \email{hugues.verdure@uit.no}}
%
%
\maketitle

\abstract{This chapter investigates how  symmetries can be used to reduce the computational complexity in polynomial optimization problems. 
A focus will be specifically given on the Moment-SOS hierarchy in polynomial optimization, where results from representation theory and invariant theory of groups can be used.
In addition, symmetry reduction techniques which are more generally applicable are also presented.}

\section{Introduction}
\begin{quote}
{\it Symmetry is a vast subject, significant in art and nature. Mathematics lies at its root,
and it would be hard to find a better one on which to demonstrate the working of the
mathematical intellect. }\vspace{0.5cm}\\
\hfill Hermann Weyl
\end{quote}
Polynomial optimization  is concerned with   optimization problems  expressed by  polynomial functions. Problems of this kind can arise in many different contexts, including engineering, finance, and computer science. Although these problems may be formulated  in a rather elementary way, they are in fact challenging and  solving  such problems is known to be algorithmically hard in general. It is therefore  beneficial to explore the algebraic and geometric structures underlying a given problem to design more efficient algorithms, and a kind of structure which is omnipresent in algebra and geometry is symmetry. In the language of algebra, symmetry is the invariance of an object or a property by some action of a group. The goal of the present chapter is to present  techniques which allow to reduce the complexity of an optimization problem with symmetry. Since it is impossible to give an exhaustive and detailed description of this vast domain, our goal here is to focus mainly on the Moment-SOS hierarchy in polynomial optimization and elaborate how the computation of such approximations can be simplified using results from representation theory and invariant theory of groups. In addition to these approaches we also mention some more general results, which allow to reduce the symmetry directly on the formulation of the polynomial optimization problem.

\begin{overview}{Overview}
 We begin first with a short presentation of  the basic Moment-SOS hierarchy in global polynomial optimization and semidefinite programming in the following Section. We do not give a very extended exposition of this topic but  limit Section~\ref{sos}  to defining the concepts  which are essential for this chapter.  Section~\ref{seq:rep}, which  is devoted to the representation theory of (finite) groups, lays out a first  set of tools allowing for reduction of  complexity. We begin in Subsection \ref{ssec:BasicRep} with a survey of the central ideas of representation theory from the point of view of symmetry reduction, where the main focus is on Schur's Lemma and its consequences. Subsection \ref{ssec:RepSn} then gives a short collection of central combinatorial objects which are used to understand the  representation theory of the symmetric group. Finally, Subsection \ref{ssec:repSDP} outlines how representation theory can be used to simplify semidefinite programs via block-diagonalization of  matrices which are commuting with the group action. Building on representation theory we then turn to invariant theory in Section~\ref{sec:invariant} which allows to closer study the specific situation of sums of squares which are invariant by a group. We begin with a basic tutorial on invariant theory in Subsection~\ref{ssec:BasInv}. Since this theory is easier  in the particular situation of finite reflection groups, we put a special emphasis on these groups. Following this introduction we show in Subsection~\ref{ssec:InvSOS} how the algebraic relationship between polynomials and invariant polynomials can be used to gain additional understanding of the structure of invariant sums of squares.  These general results are then exemplified in Subsection~\ref{ssec:Symsos} for the case of symmetric sums of squares.  Finally, Section~\ref{ssec:Misc} highlights some additional techniques for symmetry reduction: firstly, we overview in Subsection \ref{ssec:OrbSp} how rewriting the polynomial optimization problem in terms of invariants and combining it with the semialgebraic description of the orbit space of the group action can be used. In Subsection~\ref{ssec:AMGM}, we show that even in situations where Schur's Lemma from representation theory does not directly apply, one can already obtain good complexity reduction by structuring computations per orbit, and we illustrate this idea in the context of so-called SAGE certificates.  We conclude, in Subsection \ref{ssec:SymSol} with some results guaranteeing the existence of structured optimizers in the context of polynomial optimization problems with symmetries. 
\end{overview}

\section{Preliminaries on the Moment-SOS hierarchy in polynomial optimization and semidefinite programming}\label{sos}
Given real polynomials $f,g_1,\ldots,g_m\in\R[X_1,\ldots,X_n]$, a polynomial optimization problem is an optimization task of the following kind
\begin{equation}
\label{eq:opt1}
  \begin{array}{rcll}
  f^* & = & \inf\big\{ f(x)\,:\, g_1(x) \ \ge \ 0, \ldots, g_m(x) \ \ge \ 0 \, ,x\in\R^n\big\}.
  \end{array}
\end{equation}
As motivated above, the task of finding the optimal value of a such problems arises naturally in many applications. However, it is also known that this class of problems is algorithmically hard (\cite{muarty}) and a given problem might be practically impossible to solve. One approach to overcome such challenges consists in relaxing the problem: instead of solving the original hard problem, one can relax the hard conditions and define a new problem, easier to solve, and whose solution might still be close to the solution of the original problem. This idea has produced striking new ways to approximate hard combinatorial problems, such as the Max-Cut problem \cite{GW}.  One quite successful approach to relax polynomial optimization problems uses the connection of positive polynomials to moments and sums of squares of polynomials. We shortly outline this approach in the beginning since a big focus of this chapter deals with using symmetries in this setup. The overview we give here is short and a reader who is not familiar with the concepts is also advised to consult \cite{laurent-2009, lassere-book} and in particular the original works by Lasserre \cite{lasserre-2001} and Parrilo \cite{parrilo}. For simplifications we explain the main ideas only in the case of global optimization, i.e., when the additional polynomial constraints are trivial. In this case the problem we are interested to solve is to find \begin{equation}\label{pop}
f^* =\inf \left\{f(x)\,:\, x\in\R^n\right\}.\end{equation} This problem is in general a non-convex optimization problem and it can be beneficial to slightly change perspective in order to obtain the following equivalent formulations which are in fact convex optimization problems. Firstly, one can associate to a point $x\in\R^n$ the Dirac measure $\delta_x$ which leads to the following reformulation of the problem:
\begin{equation}\label{pop1}
f^*=\inf \left\{\int f(x)d\mu(x)\,:\, \int 1 d\mu(x)=1\right\},
\end{equation}
where the infimum is considered over all probability measures $\mu$ supported in $\R^n$. Since the Dirac measures $\delta_x$ of every point are feasible solutions this reformulation clearly yields the same solution.
Secondly, one can also reformulate in the following way:
\begin{equation}\label{pop2}
f^* =\sup \left\{\lambda\,:\, f(x)-\lambda\geq 0\;\forall x\in\R^n\right\}.
\end{equation}
It follows from results of Haviland \cite{haviland1936momentum} that both of the above presented reformulations are dual to each other.
Even though the second formulation is equivalent from the algorithmic perspective, it points to the core of the algorithmic hardness of the polynomial optimization problem: it is algorithmically hard to test if a given polynomial $f\in\R[X_1,\ldots,X_n]$ is non-negative, i.e., if $f(x)\geq 0$ for all $x\in\R^n$. Therefore an approach to obtain approximations for $f^*$ which are easier to calculate  consists in replacing this condition with one that is easier to check, but still ensures non-negativity. We say that a polynomial $f\in\R[X_1,\ldots,X_n]$ is a \emph{sum of squares (SOS)} if it can be written as $f=p_1^2+\ldots+p_l^2$ for some polynomials $p_1,\ldots,p_l\in\R[X_1,\ldots,X_n]$. Clearly every such polynomial is also non-negative and  we can consider the problem
\begin{equation}\label{eq:sos}
f^{SOS} =\sup \left\{\lambda\,:\, f(x)-\lambda \text{ is a SOS}\right\}.\end{equation}
Clearly we have $f^{SOS}\leq f^*$, and as not every non-negative polynomial is a sum of squares, as was shown by Hilbert \cite{hilbert1888darstellung}, this inequality will not be sharp in general. On the other hand, having a concrete decomposition into sums of squares provides an \emph{algebraic certificate} of non-negativity. So one obtains an algebraically verifiable proof for the solution of the bound.  
This strengthening of the optimization problem gives, on the dual side, a relaxation of the moment formulation.
The main feature which makes this approach also interesting for practical purposes comes from the fact that one can find a SOS-decomposition with the help of semidefinite programming. This was used by Lasserre \cite{lasserre-2001}  to define a hierarchy for polynomial optimization based on the sums of squares strengthening \eqref{eq:sos} and the corresponding relaxation of \eqref{pop1}. In the context of polynomial optimization on compact sets semi algebraic sets  this yields  (some additional assumptions)  yields a converging sequence of approximations each of which can be obtained by a semidefinite program (see \cite{laurent-2009, lassere-book} for more details on these hierarchies).

\begin{trailer}{Semidefinite programming}
Semidefinite programming (SDP) is an  optimization paradigm which was developed  in  the 1970s as a  generalization of linear programming. The main setup is to maximize/minimize a linear function over a convex set which is defined by the requirement that a certain matrix is positive semidefinite. We denote by $\Sym_n(\R)$ the set of all real symmetric $n\times n$ matrices.  Then a matrix $A\in\Sym_n(\R)$ is called \emph{positive semidefinite} if $x^{T}A x\geq 0 $ for all $x\in\R^{n}$. In this case, we write $A  \succeq 0$. Furthermore  for $A,B\in\Sym_n(\R)$ we consider their scalar product \[\langle A,B\rangle =\Tr(A\cdot B).\]
The set of all symmetric matrices $A\in\Sym_n(\R)$ which are positive semidefinite defines a convex cone inside $\Sym_n(\R)$. With these notations we can define what an SDP is.
\begin{definition}
Let  $C, A_1,\ldots, A_m\in \Sym_n(\R)$ be symmetric matrices and let $b\in \R^m$ \. Then  a semidefinite program is an  optimization problem of the form 
\[
  \begin{array}{rcl}
 \multicolumn{3}{l}{y^*:=\inf \, \langle X,C \rangle} \\
  \text{}{s.t.} \quad \langle X, A_i \rangle & = & b_i \, , \quad 1 \le i \le m \, , \\
    X & \succeq & 0 , \text{ where }X \in \Sym_n(\R) \, .
  \end{array}
\]
\end{definition}
 The \emph{feasible set} $$\mathcal{L}:=\{X\in\Sym_n(\R)\,:\,\,\langle A_i, X \rangle = b_i \, ,  1\le i \le m \, , \,X  \succeq 0\}$$ is a convex set. 
 
 The main feature which spiked the interest in this class of optimization problem is that they are on the one hand  practically solvable (see \cite{nesterov1994interior, ramana1997exact} for more details) but on the other hand can be used to design good approximation of otherwise algorithmically hard optimization problems (see for example \cite{wolkowicz2012handbook} for a detailed overview). 
 \end{trailer}
\begin{trailer}{The Lov\'asz number  of a graph}
One example in which SDPs have proven to be especially powerful is in combinatorial optimization. In fact the seminal paper \cite{lovasz79} in which   Lov\'asz  introduced the  parameter $\vartheta$ defined below as the solution to a semidefinite program was one of the first instances where the formulation of SDPs arose. The combinatorial problem for which these notions were designed is related to the following parameters of a graph.  

\begin{definition}
  Let $\G=(V,E)$ be a finite graph.
 \begin{enumerate} 
  \item A set $S\subset V$ is called \emph{independent}, if no two elements in $S$ are connected by an edge. 
  Let  ${\alpha (\G)}$ denote the cardinality of the \emph{largest independent set} in $\G$.
  \item For $k\in \N$  a \emph{$k$-coloring} of the vertices of $\G$ is a distribution of $k$ colors to the vertices of $\G$ such that two neighbouring vertices obtain different  colors.
  The number ${\chi(\G)}$ denotes the smallest $k$ such that there is a $k$-coloring of the vertices in  $\G$.
  \end{enumerate}
\end{definition}
Let $V={v_1,\ldots,v_n}$ and $S\subset V$ be any independent vertex set. We consider the characteristic function $\1_S$ of $S$. Using this function, we can construct the following $n\times n$ matrix $M$:
$$M_{i,j}=\frac{1}{|S|} \1_S(v_i)\1_S(v_j).$$
It is clear that $M\in\Sym_n(\R)$. Furthermore, clearly $M\succeq 0$. Additionally, since $\1_S$ is the characteristic function of $S$ and $S$ consists only of vertices that do not share an edge, the following three properties of the matrix $M$ hold:
\begin{enumerate}
\item $M_{i,j}=0$ if $\{v_i,v_j\}\in E,$
\item $\sum_{v_i\in V} M_{i,j}=1,$
\item $\sum_{\{v_i,v_j\}\in V\times V} M_{i,j}=|S|$.
\end{enumerate}
With these properties of $M$ in mind, one defines the $\vartheta$-number of a graph $\Gamma$.
\begin{definition}\label{def:theta}
Let $\Gamma=(V,E)$ be a graph with vertex set $V=\{v_1,v_2,\dots,v_n\}$. Then the $\vartheta$-number of $\Gamma$ is defined  as the solution to the following SDP:\begin{equation}\label{theta primal}
\begin{array}{lll} 
  \vartheta(\Gamma) & = \sup\big\{  \sum_{i,j} B_{i,j} : &
  B\in \Sym_n(\R),\ B \succeq 0 \\
&&   \sum_i B_{i,i}=1,\\
&&   B_{i,j}=0 \quad  \forall (i,j)\in E\big\}\\
\end{array}
\end{equation}
\end{definition}

Note that the above mentioned graph invariants are known to be hard to compute (see \cite{gare1974some, karp1972reducibility}). On the other hand, the $\vartheta$-number is defined as the optimal  solution of a semidefinite program and thus  easier to calculate. In fact,  Lov\'asz could show the following remarkable relationship.
\begin{theorem}[Sandwich theorem]
With the notions defined above we have:
\begin{equation*}
\alpha(\Gamma)\leq \vartheta(\Gamma)\leq \chi(\overline{\Gamma})
\end{equation*}
\end{theorem}
Furthermore, for the class of perfect graphs the  $\vartheta$-number actually provides a sharp bound. Thus, in these cases, and as $\alpha(\Gamma)$ is an integer, semidefinite programming yields polynomial time algorithms for computing these graph parameters. 
\end{trailer}
 Following  Lov\'asz's work, various different SDP approximations of hard problems have been proposed, for instance in coding theory and sphere packing (see \cite{gijswijt2006new, laurent2007strengthened, bachoc2008new, de2015semidefinite, DeLaat2022k}).
 \begin{trailer}{Connecting  SDPs with sums of squares}
In the context of polynomial optimization the relation to semidefinite approximations comes through the following observation originally due to Powers and Wörmann \cite{Powers-Wormann98}  that  one can obtain a sums of squares
decomposition of a given polynomial $f$ via the so-called Gram Matrix method, which
transfers this question into a semi-definite program. This connection is established in the following way: Let $p \in \R[X_1,\ldots,X_n]$ be a polynomial  of even degree $2d$. With a slight abuse of notation we denote by  $Y^d$ a vector  containing  all $\binom{n+d}{d}$ monomials in the variables $X_1, \ldots, X_n$ of degree at most $d$. Thus, every polynomial $s=s(X)$ of degree $d$ is uniquely determined by its coefficients  relative to $Y$. Now assume that  $p$ decomposes  into a form

\[
  p = \sum_j (s_j(X))^2 \quad \text{ with polynomials } s_j \text{ of degree at most } d .\, 
\]
Then with the above notation we can rewrite this as 

\[
 p \ = \ (Y^d)^T  \big( \sum_j s_j s_j^T \big)  Y^d \, ,
\]
where now each $s_j$ denotes the coefficient vector  of the polynomials $s_j(X)$. In this case the matrix $Q:=\sum_j s_j s_j^T$ is positive semi-definite. Since by the so called Cholesky decomposition every positive semidefinite matrix $A\in\Sym_n(\R)$ can be written in the form $A=\sum_{j}a_ja_j^T$ for some $a_j\in\R^n$ we see that the above line of argument is indeed an equivalence and so the existence of a sum of squares decomposition of $p$ follows by providing a feasible solution to a semi-definite program, i.e we have 
\begin{proposition}\label{prop:sossemidefinite}
A polynomial $p \in \R[X]$ of degree $2d$ is a sum of squares, if and only if there is a  positive semi-definite matrix $Q$ with
\[
p \ = \ Y^T Q Y.
\]\end{proposition}
With this observation the formulation  \eqref{eq:sos} can be directly transferred into the framework of semidefinite programming.  This first approximation can already yield good  bounds for the global optimum  $f^*$ and in particular in the constraint case, it can be developed further to a hierarchy of SDP--approximations: for a given optimization problem in the form \eqref{pop} which satisfies relatively general conditions on the feasible set $K$ one can construct a hierarchy of growing SDP--approximations whose optimal solutions converge towards the optimal solution of the initial problem. This approach gives a relatively general method to approximate and in some cases solve the initial problem. 
\end{trailer}
Even though the Moment-SOS formulation described above yields a computational viable way to approximate the optimal solution the dimension of the matrices in the resulting SDPs can grow fast with the problem size. This is why this approach is limited to small or medium size problems unless some specific characteristics are taken into account. The  focus of this chapter is on presenting ways to overcome this  bottleneck  in the case when additional symmetry is present in the problem.

\section{Using Representation theory in SDPs for sums-of-squares}\label{seq:rep}

Representation theory is a branch of mathematics that studies symmetries and their relation to algebraic structures. More concretely, it studies the ways in which groups can be represented as linear transformations of vector spaces  by representing the elements of the groups as invertible matrices. In this way it becomes possible to examine the structure of the group using the tools of linear algebra, making easier the study of structural algebraic  properties of the group. Representation theory has also proven to be a powerful tool to simplify computations, in the situation where a group action is assumed on a vector space of matrices. 
In this situation we can use representation theory to describe the set of  matrices stabilized by the action in a simplified way, reducing the sizes of the matrices as well as the dimension of the matrix spaces.
In turn, this allows to study other algebraic properties, such as positive semidefiniteness of the invariant matrices in a more efficient way. 
This computational aspect of representation theory gives rise to  practical applications in many fields, including physics, chemistry, computer science, and engineering. For example, it can be used to study the behavior of particles in quantum mechanics and the properties of molecules in chemistry and many more (see for example \cite{rep1,rep2,rep3,rep4,rep5}).
In this section, we outline the basic ideas of representation theory and its use in particular in semidefinite programming, where representation theory has successfully served as a key for many fascinating applications (see \cite{bachoc, vallentin} for other tutorials on the topic).
Our main focus here lies in the reduction of semi-definite optimization problems and sums of squares representations of invariant polynomials. We start by providing a basic introduction to representation theory. A comprehensive and approachable introduction to this topic can be found in Serre's book \cite{serre-b77}. Further readings on the subject can be found in \cite{stanley2011enumerative, stiefel2012group, fultonharris, james}. 
\subsection{Basic representation theory}\label{ssec:BasicRep}
We begin by introducing the central definitions of representation theory. 

\begin{definition}
Let $G$ be a group. 
\begin{enumerate}
\item A representation of $G$ is a pair $(V, \rho)$, where $V$ is a vector space over a field $\mathbb{K}$ and $\rho: G \rightarrow \operatorname{GL}(V)$ is a group homomorphism from $G$ to the group of invertible linear transformations of $V$. The degree of the representation is the dimension of the vector space $V$. 
\item Two representations $(V, \rho)$ and $(V', \rho')$ of the same group $G$ are considered equivalent, or isomorphic, if there exists an isomorphism $\phi: V \rightarrow V'$ such that $\rho'(g) = \phi \rho(g) \phi^{-1}$ for all $g \in G$.
\item Given a representation $(V, \rho)$, we can associate to it its character $\chi_V: G \rightarrow \mathbb{K}$, which is defined as \[\chi_V(g) = \Tr(\rho(g)).\] 
\end{enumerate}
\end{definition}
\begin{remark}
A character of a group  is a class function: it is constant on the conjugacy classes of $G$.
\end{remark}
If $V$ is a finite-dimensional vector space, we can identify the image of $G$ under $\rho$ as a matrix subgroup $M(G)$ of the invertible $n \times n$ matrices with coefficients in $\mathbb{K}$ by choosing a basis for $V$. We will denote the matrix corresponding to $g \in G$ as $M(g)$ and refer to the family $\{M(g) \mid g \in G\}$ as a matrix representation of $G$.

\begin{example}{Examples of representations}
To give a selection of examples, we consider the group $\S_n$ of permutations of a set $S = \{1, 2, \dots, n\}$. The group $\S_n$ has several representations:
\begin{enumerate}
    \item The \emph{trivial representation} $V = \mathbb{C}$ with $g(v) = v$ for all $g \in \S_n$ and $v \in \mathbb{C}$. This trivial representation can analogously be defined for all groups. 
    \item The \emph{natural representation} of $\S_n$ on $\mathbb{C}^n$. In this representation, $\S_n$ acts linearly on the $n$-dimensional space \[V = \mathbb{C}^n = \bigoplus_{i=1}^n \mathbb{C} e_i.\] For $g \in \S_n$, we define $g \cdot e_i = e_{g(i)}$ and extend this to a linear map on $V$. With respect to the basis $e_1,\ldots,e_n$ we obtain thus a matrix representation given by \emph{permutation matrices} which are defined as $M(g) = ( x_{ij} ) \in \C^{n \times n}$ with entries
\[x_{ij} = \delta_{g(i),j} = \begin{cases} 1, & \text{if } g(i)=j \\ 0, & \text{else.} \end{cases}\]
\item A representation of $\S_n$ on the polynomial ring $\mathbb{C}[X_1, \dots, X_n]$. In this representation,  we define  \[g \cdot X_i = X_{g(i)}\text{ for }g \in \S_n \text{ and } i \in {1, \dots, n},\] and extend it to a morphism of $\C$-algebras $\C[X_1,\cdots,X_n] \longrightarrow \C[X_1,\cdots,X_n]$. In this case we write $f^g$ for the image of a polynomial $f$ under the action of $\sigma$. This notation analogously applies to more general groups. 

    \item The regular representation, where $(\S_n,\circ)$ acts on the vector space \[\mathbb{C}{\S_n} = \bigoplus_{s \in \S_n} \mathbb{C} e_s\] with formal symbols $e_s$ ($s \in \S_n$) via $g(e_s) = e_{g \circ s}$.  Also, this regular representation can analogously be defined for  other finite  groups.
\end{enumerate}
\end{example}
\begin{trailer}{The notion of $G$-module}
Another way to define representations is to consider the linear action of a group $G$ on a vector space $V$ which gives $V$ the structure of a $G$-module. A $G$-module is an abelian group $M$ on which the action of $G$  respects the abelian group structure on M. In the case when $M$ is a vector space $V$, this just means that the group action is compatible with the operations on $V$. Given a vector space $V$ that is a $G$-module, we can define a map $\phi: G \rightarrow \operatorname{GL}(V)$ that sends an element $g \in G$ to the linear map $v \mapsto gv$ on $V$. This map $\phi$ is then a representation of $G$. Thus in the context of linear spaces, the notions of $G$-modules and representations of $G$ are equivalent, and it can be more convenient to use the language of $G$-modules. In this case, we call a linear map \[\phi: V \rightarrow W\] between two $G$-modules a \emph{$G$-homomorphism} if \[\phi(g(v)) = g(\phi(v))\text{ for all }g \in G\text{ and }v \in V.\] The set of all $G$-homomorphisms between $V$ and $W$ is denoted by $\operatorname{Hom}_G(V, W)$. Two $G$-modules are considered \emph{isomorphic (or equivalent)} as $G$-modules if there exists a $G$-isomorphism from $V$ to $W$.
\end{trailer}
Given a representation $V$, the set of all $G$--homomorphisms from $V$ to itself is called the \emph{endomorphism algebra} of $V$. It is denoted by $\operatorname{End}(V)$. If we have a matrix representation $M(G)$ of a group $G$, the endomorphism algebra corresponds to the \emph{commutant algebra} $\operatorname{Com}((M(G))$. This is the set of all matrices commuting with the group action, i.e., 
\[\operatorname{Com}(M(G)):=\{T \in \mathbb{C}^{n \times n} \text{ such that } TM(g) = M(g)T \text{ for all }g \in G\}.\] When studying the action of a group $G$ on a vector space $V$, it is important to consider subspaces that are closed under the action. Such a subspace is called a \emph{subrepresentation} or \emph{$G$-submodule} of $V$. If a representation $V$ has a proper submodule, it is called \emph{reducible}. If the only submodules of $V$ are $V$ and the zero subspace, the representation is called \emph{irreducible}.

\begin{example}{Decomposition of the natural representation of $\S_n$}
Consider the natural representation of $\S_n$ on the $n$-dimensional vector space $\C^n$. This representation is not irreducible. It has in fact two non-trivial subrepresentations, namely \[W_1 = \C \cdot (e_1+e_2+\ldots+e_n) \quad \text{ and } \quad W_2 = W_1^\perp\] the orthogonal complement of $W_1$ for the usual inner product. Clearly, $W_1$ is irreducible since it is one-dimensional. Furthermore, also $W_2$ is in fact  irreducible.
\end{example}
In the previous example, the vector space $V$ is an orthogonal sum of irreducible representations. This is generally the case when the order of the group is not divisible by the characteristic of the ground field. 
In this case,  a $G$-invariant inner product can be  defined starting from any inner product $\langle\cdot ,\cdot \rangle$ via
\[
\langle x,y\rangle_G=\frac{1}{|G]}\sum_{g\in G}\langle g(x) ,g(y)\rangle.
\] The existence of such a $G$--invariant scalar product is in fact the key to the following Theorem.
 
\begin{theorem}[Maschke's Theorem]\label{thm:Mashcke}
Let $G$ be a finite group and  consider a representation $(V,\phi)$ defined over a field with a characteristic which is prime to $|G|$. Then, for every subrepresentation $W$ we have that the orthogonal complement of $W$ with respect to $\langle\cdot, \cdot\rangle_G$ is also a subrepresentation. 
\end{theorem}
\begin{trailer}{Remarks on Maschke's Theorem}
 As a consequence of Maschke's Theorem, every reducible representation $V$ can be decomposed as an orthogonal sum of irreducible representations. Note that such a decomposition is not necessarily unique. Indeed, if we consider a  group $G$ which acts trivially on a $n$-dimensional vector space, every one-dimensional subspace is an irreducible subrepresentation. Furthermore, the assumption that the group is finite is not necessary and  a $G$-invariant inner product also exists in the case of compact  groups. So,  a version of Maschke's theorem  also holds, for example, for compact groups, like $O(n)$. 
\end{trailer}

A very central statement in representation theory, which is, in particular, the key to the symmetry reductions in the further sections, is the following statement which goes back to Schur. 
\begin{theorem}[Schur's Lemma]\label{thm:schur}
Let $V$ and $W$ be two irreducible $G$-modules  of a group $G$.  Every $G$-endomorphism from $V$ to $W$ is either zero or invertible. Furthermore, $\End(V)$ is a skew field over~$\K$. In particular, 
if the ground field $\K$ is algebraically closed, every $G$ invariant endomorphism between $V$ and $V$ is a scalar multiple of the identity. 
\end{theorem}
\begin{corollary}\label{cor:Schur}
Let $V$ be a complex irreducible $G$-module  and $\langle\cdot,\cdot\rangle_G$ be an
invariant Hermitian form on $V$. Then $\langle\cdot,\cdot\rangle_G$ is unique up to a real scalar multiple.
\end{corollary}
One can define an inner product on the set of complex valued functions on a finite group $G$ by \[\langle \phi, \psi \rangle = \frac{1}{|G|} \sum_{g \in G} \phi(g) \overline{\psi(g)}. \]
Using Schur's Lemma, one can show the following important property of irreducible characters.
\begin{theorem}
Let $G$ be a finite group and $\K$ be algebraically closed. Then, for every pair of characters $(\chi_i,\chi_j)$ corresponding to a pair of non-isomorphic irreducible representations  we have  \[\langle \chi_i,\chi_j\rangle =0.\] Moreover,  the set of irreducible characters $\{\chi_i,\, i\in I\}$ form an orthonormal basis for the $\K$-vector space of  class functions of $G$.\end{theorem}
Recall that a \emph{class function} on $G$ is a function $G\rightarrow \K$ which is constant on the different conjugacy classes of $G$. Therefore, the following is a direct consequence. 
\begin{corollary} Assume that $G$ is finite. Then there exists a finite number of non-equivalent irreducible representations of $G$. If $\K$ is algebraically closed, this  number is equal to the number of conjugacy classes of $G$.
\end{corollary}
Motivated by the notion of irreducible representations is the following notion of the \emph{isotypic decomposition} of a representation. Here, isotypic means that we aim to decompose a representation $V$ as a direct sum of subrepresentaions, where each summand is the direct sum of equivalent irreducible subrepresentations.
Specifically, for a finite group $G$, suppose that $\mathcal{I}=\{W_1, W_2, \dots, W_k\}$ are all the isomorphism classes of irreducible representations of $G$. Then the isotypic decomposition of $V$ is a decomposition \[V = \bigoplus_{i=1}^k\bigoplus_{j=1}^{m_i} V_{ij},\] where each $V_{ij}$ is a subrepresentation of $V$ that is equivalent to the irreducible representation $W_i$. In other words, for each $i$, there exist   $G$-isomorphisms \[\phi_{i,j}: W_i \rightarrow V_{ij} \/, \text{ for all }1\leq j\leq m_i\] The subspace \[V_i = \bigoplus_{j=1}^{m_i} V_{ij}\] is called the isotypic component associated of type $W_i$. Since there is a bijection between the set of characters and the set of isomorphism classes of representations, we can also index the isotypic components of a representation by the corresponding irreducible character. We then denote by $V^{\chi}$ the isotypic component of $V$ associated to the irreducible representation of character $\chi$. The resulting   isotypic decomposition of $V$ \[V=\bigoplus_{i=1}^k V_i = \bigoplus_{\chi} V^\chi\] is unique up to the ordering of the direct sum.
\begin{backgroundinformation}{Computing an isotypic decomposition}
Using the characters of the irreducible representations, it is possible to calculate the isotypic decomposition as images of the projections \begin{equation}\label{eq:1}\abb{\pi_i}{V}{V}{x}{\frac{\dim W_i}{\mid G \mid} \sum_{g \in G} \overline{\chi_i(g)} g \cdot x}\end{equation} 
where $\chi_i$ is the character associated to the irreducible representation $W_i$.
\end{backgroundinformation}
Combining the isotypic decomposition with Schur's Lemma, we obtain the following for representations defined over $\C$.
\begin{corollary}\label{cor:decomp}
Let $V:=m_1W_1\oplus m_2 W_2\oplus\ldots\oplus m_k W_k$ be a complete decomposition of a representation $V$ over $\C$ such that $\dim W_i=d_i$. Then we have: 
\begin{enumerate}
\item $\dim V=m_1d_1+\ldots m_kd_k$, 
\item $\End V\simeq \bigoplus_{i=1}^k\C^{m_i\times m_i}.$ 
\item Let $\chi$ be the character of $V$ and $\chi_i$ the character of $W_i$ then we have
$\langle \chi,\chi_i\rangle=m_i$.
\item There is a basis of $V$ such that 
	\begin{enumerate}
	\item The matrices of the corresponding matrix group $M(G)$ are of the form $$M(g)=\bigoplus_{l=1}^{k}\bigoplus_{j=1}^{m_i}M^{(l)}(g),$$ where $M^{(l)}(G)$ is a matrix representation of $G$ corresponding to $W_l$.
	\item The corresponding commutant algebra is of the form \[\Com M(G)\simeq\bigoplus_{l=1}^k(N_{i}\otimes I_{d_l}),\] where $N_l\in \C^{m_l\times m_l}$ and $I_{d_l}$ denotes the identity in $\C^{d_l \times d_l}$.
	\end{enumerate}
\end{enumerate}
\end{corollary}
A basis for $V$ as in the corollary above is called \emph{symmetry adapted} basis. Given a matrix representation $X$ of a representation $V$ of a group $G$, Corollary \ref{cor:decomp} amounts to construct a basis of $V$ such that the corresponding matrix representation has a particularly simple form. Specifically, we can decompose $V$ into a direct sum of subrepresentations $V_{l,\beta}$ that are isomorphic to an irreducible representation $W_l$ of $G$, and the matrices of the representation will be block diagonal with blocks corresponding to these subrepresentations. 
\begin{backgroundinformation}{How to construct a symmetry adapted basis}
To construct such a  basis, we consider a matrix representation $Y^l$ corresponding to an irreducible representation $W_l$ of dimension $d_l$ and define the map \[\pi_{\alpha,\beta}: V \rightarrow V \text{ for each } \alpha,\beta = 1,\ldots, d_l,\]  as 
\[\pi_{\alpha,\beta} = \frac{m_l}{|G|}\sum_{g\in G}Y^l_{\beta,\alpha}(g^{-1})M(g).\]
Here, $m_l$ is the number of copies of $W_l$ that appear in the decomposition of $V$. It can be shown (see \cite[Section 2.6]{serre-b77}) that the map $\pi_{1,1}$ is a projection from $V$ onto a subspace $V_{l,1}$ isomorphic to $W_l$, and $\pi_{1,\beta}$ maps $V_{l,1}$ onto $V_{l,\beta}$, which is another subrepresentation of $V$ isomorphic to $W_l$. Thus we arrive at the announced decomposition and can use the maps to construct a symmetry-adapted basis.
\end{backgroundinformation}
A representation that admits a very beautiful isotypic decomposition is the \emph{regular  representation} of a finite group $G$. Recall that this is defined  on  the vector space \[V^{reg} = \bigoplus_{g \in G} \C e_g,\] via  $\rho(g)(e_h) = e_{g\cdot h}$ for every $g,h \in G$. 

\begin{theorem}\label{them:regular}
Let $G$ be a finite group and $(V,\rho)$  isomorphic to the regular representation of $G$.
Then,\[V \simeq \bigoplus_{W \in \mathcal{I}} (\dim W) W,\] and in particular, we have \[|G| = \dim V = \sum_ {W \in \mathcal{I}} (\dim W)^2.\]
\end{theorem}
\begin{example}{Cyclic permutation matrices and the associated commutant}
Consider the cyclic group $C_4$ and let $g$ be a generating element of this group, i.e., $C_4=\{g^0,g^1,g^2,g^3\}$. The regular representation of this group is of dimension 4.
It can be defined as a matrix representation via the \emph{cyclic permutation matrices} given via
\[
g\mapsto M^{reg}(g)=\begin{pmatrix}
          0&1&0&0\\
          0&0&1&0\\
          0&0&0&1\\
          1&0&0&0
          \end{pmatrix}.
\]
Furthermore, $C_4$ has 4 pairwisely non-isomorphic irreducible representations, each of which is one-dimensional and we get these representations via  \[\begin{array}{cccc}\rho_j:& G & \longrightarrow & \C \\ & g & \longmapsto & e^{\frac{2\pi  \mathrm{i}}{4} j}\end{array}\] for $0 \leqslant i \leqslant 3$.
With the projection defined above in \eqref{eq:1} we obtain that the symmetry adapted basis \[B:=\left\{(1,1,1,1),(1,\mathrm{i},-1,-\mathrm{i}),(1,-1,1,-1),(1,-\mathrm{i},-1,\mathrm{i})\right\}.\] With respect to this basis, the representation $V^{reg}$ is given via the diagonal matrix 
\[g\mapsto \tilde{X}^{reg}(g)=\begin{pmatrix} 1&\phantom{-}0&\phantom{-}0&\phantom{-}0\\0&\phantom{-}\mathrm{i}&\phantom{-}0&\phantom{-}0\\0&\phantom{-}0&-1&\phantom{-}0\\0&\phantom{-}0&\phantom{-}0&-\mathrm{i}\end{pmatrix}\]
Now, consider a \emph{circulant matrix}, i.e., a matrix of the from \[T:=  \left( \begin {array}{cccc} \alpha&\beta&\gamma&\delta\\ \noalign{\medskip}\beta&\gamma&\delta&\alpha
\\ \noalign{\medskip}\gamma&\delta&\alpha&\beta\\ \noalign{\medskip}\delta&\alpha&\beta&\gamma\end {array}
 \right).
\]
Clearly, this matrix commutes with the matrix representation $X^{reg}$ and doing the change of basis to the symmetry adapted basis we obtain

\[\tilde{T}= \left( \begin {array}{cccc} \alpha+\beta+\gamma+\delta&0&0&0\\ \noalign{\medskip}0&\alpha+\mathrm{i}\beta-
\gamma-\mathrm{i}\delta&0&0\\ \noalign{\medskip}0&0&\alpha-\beta+\gamma-\delta&0\\ \noalign{\medskip}0&0&0&\alpha
-\mathrm{i}\beta-\gamma+\mathrm{i}\delta\end {array} \right) 
\]
More generally, for the cyclic group $C_n$  we see that this regular representation will contain $n$ irreducible representations $\rho_0,\ldots, \rho_{n-1}$, all of which are 1-dimensional and given by
\[
\rho_j:\,g\mapsto\, e^{\frac{2\pi \mathrm{i}}{n}j},
\]
and  corresponding symmetric adapted bases is also known as the \emph{Fourier-basis}.
\end{example}                                                                                                                                                                                                                                                                                                                                                                                                                                                                                                                                                                                                                                                                                                                                                                                                                                                                                                                                                                                                                                                                                                                                                                                                                                                                                                                                                                                                                                                                                                                                                                                                                                                                                                                                                                                                                                                                                                                                                                                                                                                                                                                                                                                                                                                                                                                                                                                                                                                                                                                                                                                                                                                                                                                                                                                                                                                                                                                                                                                                                                                                                                                                                                                                                                                                                                                                                                                                                                                                                                                                                                                                                                                                                                                                                                                                                                                                                                                                                                                                                                                                                                                                                                                                                                                                                                                                                                                                                                                                                                                                                                                                                                                                                                                                                                                                                                                                                                                                                                                                                                                                                                                                                                                                                                                                                                                                                                                                                                                                                                                                                                                                                                                                                                                                                                                                                                                                                                                                                                                                                                                                                                                                                                                                                                                                                                                                                                                                                                                                                                                                                                                                                                                                                                                                                                                                                                                                                                                                                                                                                                                                                                                                                                                                                                                                                                                                                                                                                                                                                                                                                                                                                                                                                                                                                                                                                                                                                                                                                                                                                                                                                                                                                                                                                                                                                                                                                                                                                                                                                                                                                                                                                                                                                                                                                                                                                                                                                                                                                                                                                                                                                                                                                                                                                                                                                                                                                                                                                                                                                                                                                                                                                                                                                                                                                                                                                                                                                                                                                                                                                                                                                                                                                                                                                                                                                                                                                                                                                                                                                                                                                                                                                                                                                                                                                                                                                                                                                                                                                                                                                                                                                                                                                                                                                                                                                                                                                                                                                                                                                                                                                                                                                                                                                                                                                                                                                                                                                                                                                                                                                                                                                                                                                                                                                                                                                                                                                                                                                                                                                                                                                                                                                                                                                                                                                                                                                                                                                                                                                                                                                                                                                                                                                                                                                                                                                                                                                                                                                                                                                                                                                                                                                                                                                                                                                                                                                                                                                                                                                                                                                                                                                                                                                                                                                                                                                                                                                                                                                                                                                                                                                                                                                                                                                                                                                                                                                                                                                                                                                                                                                                                                                                                                                                                                                                                                                                                                                                                                                                                                                                                                                                                                                                                                                                                                                                                                                                                                                                                                                                                                                                                                                                                                                                                                                                                                                                                                                                                                                                                                                                                                                                                                                                                                                                                                                                                                                                                                                                                                                                                                                                                                                                                                                                                                                                                                                                                                                                                                                                                                                                                                                                                                                                                                                                                                                                                                                                                                                                                                                                                                                                                                                                                                                                                                                                                                                                                                                                                                                                                                                                                                                                                                                                                                                                                                                                                                                                                                                                                                                                                                                                                                                                                                                                                                                                                                                                                                                                                                                                                                                                                                                                                                                                                                                                                                                                                                                                                                                                                                                                                                                                                                                                                                                                                                                                                                                                                                                                                                                                                                                                                                                                                                                                                                                                                                                                                                                                                                                                                                                                                                                                                                                                                                                                                                                                                                                                                                                                                                                                                                                                                                                                                                                                                                                                                                                                                                                                                                                                                                                                                                                                                                                                                                                                                                                                                                                                                                                                                                                                                                                                                                                                                                                                                                                                                                                                                                                                                                                                                                                                                                                                                                                                                                                                                                                                                                                                                                                                                                                                                                                                                                                                                                                                                                                                                                                                                                                                                                                                                                                                                                                                                                                                                                                                                                                                                                                                                                                                                                                                                                                                                                                                                                                                                                                                                                                                                                                                                                                                                                                                                                                                                                                                                                                                                                                                                                                                                                                                                                                                                                                                                                                                                                                                                                                                                                                                                                                                                                                                                                                                                                                                                                                                                                                                                                                                                                                                                                                                                                                                                                                                                                                                                                                                                                                                                                                                                                                                                                                                                                                                                                                                                                                                                                                                                                                                                                                                                                                                                                                                                                                                                                                                                                                                                                                                                                                                                                                                                                                                                                                                                                                                                                                                                                                                                                                                                                                                                                                                                                                                                                                                                                                                                                                                                                                                                                                                                                                                                                                                                                                                                                                                                                                                                                                                                                                                                                                                                                                                                                                                                                                                                                                                                                                                                                                                                                                                                                                                                                                                                                                                                                                                                                                                                                                                                                                                                                                                                                                                                                                                                                                                                                                                                                                                                                                                                                                                                                                                                                                                                                                                                                                                                                                                                                                                                                                                                                                                                                                                                                                                                                                                                                                                                                                                                                                                                                                                                                                                                                                                                                                                                                                                                                                                                                                                                                                                                                                                                                                                                                                                                                                                                                                                                                                                                                                                                                                                                                                                                                                                                                                                                                                                                                                                                                                                                                                                                                                                                                                                                                                                                                                                                                                                                                                                                                                                                                                                                                                                                                                                                                                                                                                                                                                                                                                                                                                                                                                                                                                                                                                                                                                                                                                                                                                                                                                                                                                                                                                                                                                                                                                                                                                                                                                                                                                                                                                                                                                                                                                                                                                                                                                                                                                                                                                                                                                                                                                                                                                                                                                                                                                                                                                                                                                                                                                                                                                                                                                                                                                                                                                                                                                                                                                                                                                                                                                                                                                                                                                                                                                                                                                                                                                                                                                                                                                                                                                                                                                                                                                                                                                                                                                                                                                                                                                                                                                                                                                                                                                                                                                                                                                                                                                                                                                                                                                                                                                                                                                                                                                                                                                                                                                                                                                                                                                                                                                                                                                                                                                                                                                                                                                                                                                                                                                                                                                                                                                                                                                                                                                                                                                                                                                                                                                                                                                                                                                                                                                                                                                                                                                                                                                                                                                                                                                                                                                                                                                                                                                                                                                                                                                                                                                                                                                                                                                                                                                                                                                                                                                                                                                                                                                                                                                                                                                                                                                                                                                                                                                                                                                                                                                                                                                                                                                                                                                                                                                                                                                                                                                                                                                                                                                                                                                                                                                                                                                                                                                                                                                                                                                                                                                                                                                                                                                                                                                                                                                                                                                                                                                                                                                                                                                                                                                                                                                                                                                                                                                                                                                                                                                                                                                                                                                                                                                                                                                                                                                                                                                                                                                                                                                                                                                                                                                                                                                                                                                                                                                                                                                                                                                                                                                                                                                                                                                                                                                                                                                                                                                                                                                                                                                                                                                                                                                                                                                                                                                                                                                                                                                                                                                                                                                                                                                                                                                                                                                                                                                                                                                                                                                                                                                                                                                                                                                                                                                                                                                                                                                                                                                                                                                                                                                                                                                                                                                                                                                                                                                                                                                                                                                                                                                                                                                                                                                                                                                                                                                                                                                                                                                                                                                                                                                                                                                                                                                                                                                                                                                                                                                                                                                                                                                                                                                                                                                                                                                                                                                                                                                                                                                                                                                                                                                                                                                                                                                                                                                                                                                                                                                                                                                                                                                                                                                                                                                                                                                                                                                                                                                                                                                                                                                                                                                                                                                                                                                                                                                                                                                                                                                                                                                                                                                                                                                                                                                                                                                                                                                                                                                                                                                                                                                                                                                                                                                                                                                                                                                                                                                                                                                                                                                                                                                                                                                                                                                                                                                                                                                                                                                                                                                                                                                                                                                                                                                                                                                                                                                                                                                                                                                                                                                                                                                                                                                                                                                                                                                                                                                                                                                                                                                                                                                                                                                                                                                                                                                                                                                                                                                                                                                                                                                                                                                                                                                                                                                                                                                                                                                                                                                                                                                                                                                                                                                                                                                                                                                                                                                                                                                                                                                                                                                                                                                                                                  

\begin{trailer}{Complex irreducible versus real irreducible}

The statements in Corollary \ref{cor:decomp} rely on the ground field to be algebraically closed. Therefore, a little bit of caution is necessary when working, for example, over the real numbers. Since this will be important for optimization we briefly highlight this situation. For a real irreducible representation $(V,\rho)$, there are two possible situations that lead to three different types (see~\cite[Section~13.2]{serre-b77}):
\begin{enumerate}
    \item If the complexification $V\otimes \C$ is also irreducible (type I), then  representation can be directly transferred from $V \otimes \C$ to $V$ and we can directly apply Corollary \ref{cor:decomp}.
    \item If the complexification $V\otimes \C$ is reducible, then it will decompose into two complex-conjugate irreducible $G$-submodules $V_1$ and $V_2$. These submodules may be non-isomorphic (type II) or isomorphic (type III). 
    \end{enumerate}

 If the complexification $V \otimes \C$ of a real $G$-module $V$ has the isotypic decomposition \[V\otimes\C=V_1\oplus\cdots\oplus V_{2l}\oplus V_{2l+1} \oplus\cdots\oplus V_h,\] where each pair $(V_{2j-1}, V_{2j})$ is complex conjugate ($1 \le j \le l$) and $V_{2l+1},\ldots, V_{h}$ are real, we can keep track of this decomposition in the real representation, in the following way: consider a pair of complex conjugated $G$-modules $(V_{2j-1}, V_{2j})$ with $d=\dim V_{2j-1}$. Then a basis $\mathcal{B}_{{2j-1}}$ of $V_{2j-1}$ and the conjugated basis $\mathcal{B}_{{2j}}=\overline{\mathcal{B}_{{2j-1}}}$ of $V_{2j}$  can be used to  obtain a real basis of $V_{2j-1}\oplus V_{2j}$ by considering \[\left\{b_1+b_1',\ldots,b_d+b_d',\frac{1}{\mathrm{i}}(b_1-b_1'),\ldots,\frac{1}{\mathrm{i}}(b_d-b_d')\right\}\] where $b_i\in\mathcal{B}_{{2j-1}}$ and $b_i'\in\mathcal{B}_{{2j-1}}$. Therefore, with a slight abuse of notation, we can translate the  isotypic decomposition of $V\otimes\C$ above to a  decomposition into real irreducible $G$ representations via \[V \ = \ \left( V_1+V_2 \right) \oplus \frac{1}{\mathrm{i}} \Big(V_1-V_2 \Big)\oplus \cdots \oplus \left( V_{2l-1}+V_{2l} \right) \oplus \frac{1}{\mathrm{i}} \Big( V_{2l-1} - V_{2l} \Big)\oplus V_{2l+1}\oplus\cdots\oplus V_{h.}\]  Note that in the case of a real irreducible representation $W$, also the structure of the corresponding endomorphism algebras differs from the complex case, depending on which of the three cases we are in. Indeed, in the case  of non-algebraically closed fields, the second statement in Schur's Lemma \ref{thm:schur} only yields that $\End(W)$ is isomorphic to a skew field. There are exactly three skew fields over $\R$, namely $\R$ itself, $\C$,  and the Quaternions $\HH$, and these three cases exactly correspond to the types discussed above. 
Therefore, in the case of a real irreducible representation,   we get that the endomorphism algebra is isomorphic to $\R$ if it is of  (type I), it is isomorphic to $\C$ if we are of (type II) and it is isomorphic to $\HH$ in the case of (type III).
\end{trailer}

\begin{example}{Real symmetry adapted basis for circulant matrices}
We consider again the   cyclic group $C_4$ acting  linearly on $\R^4$ by cyclically permuting coordinates. This representation is isomorphic to the regular representation, and we know a complex symmetry-adapted basis. If we denote by 
$b^{(1)}, \ldots, b^{(4)}$ the basis elements of the symmetry-adapted basis given above, then the  4 real vectors
\[\mathcal{B}:=\left\{ b^{(1)},\, b^{(3)},\, b^{(2)}+b^{(4)},\, \frac{1}{\mathrm{i}}(b^{(2)} - b^{(4)})\right\}\] yield a decomposition into real irreducible representations.
The matrix $T$ considered in the example above is then of the form

\[T_{\mathcal{B}}= \left( \begin {array}{cccc} \alpha+\beta+\gamma+\delta&0&0&0\\ 
\noalign{\medskip}0&\alpha-\beta+\gamma-
\delta&0&0\\
\noalign{\medskip}0&0&\delta-\beta&-\alpha+\gamma\\
\noalign{\medskip}0&0&\alpha-\gamma&\delta
-\beta\\ 
\end {array} \right).
\]
We thus see that the two non-real irreducible representations of $C_4$ give one real irreducible representation. We are in (type II) as these two complex irreducible representations are not isomorphic and therefore the corresponding endomorphism algebra decomposes into the blocks as above. Note that in the case of a \emph{symmetric circulant matrix} we would have obtained the same diagonal form as in the complex case. 
\end{example}
\subsection{Representation theory of $\S_n$}\label{ssec:RepSn}
The representation theory of the symmetric group, which is of particular importance due to its historical significance and connections to combinatorics, is outlined here for future reference. It is known that the conjugacy classes of the symmetric group $\S_n$ correspond one-to-one with partitions of $n$, which are non-decreasing sequences of positive integers that sum to $n$. 
A Young diagram, corresponding to a partition $\lambda=(\lambda_1,\cdots,\lambda_k) \vdash n$, is a collection of cells arranged in left-aligned rows, with $\lambda_1$ cells in the top row, $\lambda_2$ cells in the row below, and so on. 
\begin{example}{Example of a Young diagram}
Consider the partition $(4,3,1) \vdash 8$. To this partition we associate the following Young diagram: \[\ydiagram{4,3,1}\] 
\end{example}
A \emph{Young tableau} of shape $\lambda$ is obtained by filling the cells of the Young diagram of $\lambda$ with the integers $1$ to $n$. Two tableaux are considered equivalent if their corresponding rows contain the same integers. Given a tableau $T$,  its row-equivalent class $\{T\}$ can be visualized by removing the vertical lines separating the boxes in each row. Such a row equivalence class is also called a tabloid. We call the formal $\K$-vector space $\mathcal{M}^\lambda$ which is spanned by all $\lambda$-tabloids the \emph{partition module} associated to $\lambda$. 
\begin{example}{Example of equivalent Young tableaux}
Continuing with the above example and the partition $(4,3,1)\vdash 8$ we have, for example, the following two equivalent tableaux
\[T_1=\begin{ytableau}1 & 5& 7& 3 \\ 6&2&4\\8\end{ytableau} \quad \textrm{ and } \quad T_2=\begin{ytableau}7 & 5& 3& 1 \\ 2&4&6\\8\end{ytableau}.\] We can represent the row equivalence class to which these two tableaux belong by
\[\{T_1\}=\{T_2\}=\ytableausetup{tabloids}  \begin{ytableau} 3 & 1 & 7& 5 \\ 4&6&2 \\ 8\end{ytableau}.\]
Combinatorics reveals further that the associated permutation module $\mathcal{M}^{(4,3,1)}$ is spanned by the $\frac{8!}{4!\cdot 3!}=280$ different tabloids.\end{example}

\ytableausetup{notabloids}
Given a tableau $T$, we denote its columns by $C_1,\cdots,C_c$ and consider the column stabilizer subgroup $\CStab_{T} \subset \S_n$ defined by \[\CStab_{T}= \S_{C_1}\times \S_{C_2}\times\cdots\times \S_{C_c}.\] 

This setup of notations allows to define the following class of $\S_n$ representations, which turn out to give a complete list of irreducible representations in the case $\operatorname{char}(\K)=0$.

\begin{definition}\label{def:Specht}
Let $\lambda \vdash n$ be a partition. For a Young tableau $T$ of shape $\lambda$ we define 
$$E_T:=\sum_{\sigma \in \CStab_{T}}\sgn(\sigma) \{\sigma T\}.$$
Then, the Specht module $W^\lambda\subseteq\mathcal{M}^\lambda$ associated to $\lambda$ is the $\K$-vector space spanned by the $E_T$ corresponding to all Young tableaux of shape $\lambda$.
\end{definition}

A tableau is standard if every row and every column is filled in increasing order and it can be shown that the set of Young tabloids corresponding to standard Young tableaux is a minimal generating set of a Specht module. More importantly we have the following theorem:

\begin{theorem}
If $\operatorname{char}(\K)=0$ then the set $\{W^\lambda,\ \lambda \vdash n\}$ is the set of non-isomorphic irreducible representations of $\S_n$.
\end{theorem}

\begin{example}{The Specht modules of $\S_3$}\ytableausetup{tabloids, smalltableaux, centertableaux}
 The Specht modules, and hence the irreducible representations of $\S_3$, are the following three.
\begin{eqnarray*}
W^{(3)} &=& <\begin{boldmath}{\begin{ytableau}1 & 2 & 3\end{ytableau}}\end{boldmath}>,\\ 
 W^{(2,1)} &=& \left<\begin{boldmath}{\begin{ytableau}1 & 2 \\ 3\end{ytableau}} - {\begin{ytableau}3 & 2 \\ 1\end{ytableau}}\ ,\ {\begin{ytableau}1 & 3 \\ 2\end{ytableau}} - {\begin{ytableau}2 & 3 \\ 1\end{ytableau}} \end{boldmath}\right>,\\
 W^{(1,1,1)} &=& \left<\begin{boldmath}\begin{ytableau}1\\2\\3\end{ytableau}-\begin{ytableau}2\\1\\3\end{ytableau}-\begin{ytableau}3\\2\\1\end{ytableau}-\begin{ytableau}1\\3\\2\end{ytableau}+\begin{ytableau}2\\3\\1\end{ytableau}+\begin{ytableau}3\\1\\2\end{ytableau}\end{boldmath}\right>.\\
\end{eqnarray*}
We further give the associated characters. These are constant on the 3 conjugacy classes of $\S_3$, namely
$C_1=\{\Id\}$, $C_2=\{(1 2), (1 3),(2 3)\}$ and  $C_3=\{(1 2 3),(1 3 2)\}$.
\[\begin{array}{c|ccc}
&\ C_1\  &\ C_2\  &\ C_3\   \\\hline
\chi_{(3)}& 1 &\phantom{-} 1 & \phantom{-} 1 \\
\chi_{(2,1)} & 2 &\phantom{-} 0 &-1\\
\chi_{(1,1,1)} & 1 & -1 &\phantom{-} 1 \\
\end{array}\]
\end{example}
This understanding of the irreducible representations allows us to examine the following example.

\begin{example}{Diagonalization of a $S_3$-invariant Gram matrix}
We consider the permutation action of the group $\S_3$ on $\R^3$ and construct a Gram Matrix of the invariant scalar product. Since it is supposed to be $\S_3$-invariant, we find
\[\langle e_j,e_j\rangle_{\S_3}=\alpha \text{ for } j=1..3, \text{ and } \langle e_i,e_j\rangle_{\S_3}=\beta \text{ for } j\neq i .\]
The associated Gram matrix is therefore of the form
\[
\begin{pmatrix}\alpha&\beta&\beta\\ \beta&\alpha&\beta
\\ \beta&\beta&\alpha\\ 
\end{pmatrix}.\]
We further have seen that this representation decomposes into two irreducible ones, one of them being the trivial one, the other one being $W^{(2,1)}$. Thus, we find that, in this situation, the vectors $\{e_1+e_2+e_3,e_1-e_2,e_1-e_3\}$ form a symmetry adapted basis. Indeed, with respect to this basis, the Gram matrix is of the form
\[
\begin{pmatrix}
a+2\,b&0&0\\ 
0&a-b&0\\ 
0&0&a-b
\end{pmatrix}.
\]
\end{example}
\subsection{Using representation theory to simplify semidefinite formulations}\label{ssec:repSDP}

We have seen that Schur's Lemma allows for a block-diagonalization for matrices which commute with a given group action.
Now assume that $(\R^n,\rho)$ is an $n$-dimensional representation of a finite group $G$. As we can always choose an orthonormal basis for $\R^n$ with respect to a $G$-invariant scalar product, we can assume without loss of generality that the corresponding matrices are unitary, i.e we have $M(g)M(g)^T=\Id$  for all $g\in G$. Now this representation naturally carries over to a representation  on $\Sym_n(\R)$ via 
\[X^{g}:= M(g) XM(g)^{T},\quad \mbox{ for } X\in \Sym_n(\R) \,\text{and} \,g\in G.\]
A set $\mathcal{L}\subseteq \Sym_n(\K)$  is  called \emph{invariant with 
respect to $G$} if for all $X\in \mathcal{L}$ we have $X^{g}\in 
\mathcal{L},\,\mbox{ for all }\, g\in G$.
A linear functional $\langle C, X\rangle$ is \emph{$G$-invariant}, if $\langle X^{g},C\rangle=\langle X,C\rangle$ for all $g \in G$ and
 an SDP is \emph{$G$-invariant} if both the cost function $\langle X,C \rangle$ as well as the feasible set $\mathcal{L}$ are $G$-invariant.
\begin{definition}
For a given SDP  we consider 
\[
  \begin{array}{rcl}
  \multicolumn{3}{l}{y^{*}_G:=\inf \, \langle X,C \rangle} \\
  \text{s.t.} \quad \langle X,A_i \rangle & = & b_i \, , \quad 1 \le i \le m \, , \\
    {X}&  {=}&  {X^g\quad \mbox{ for all } g\in G},\\
    X & \succeq & 0 ,\text{ where }X \in \Sym_n(\R) \, .
  \end{array}
\]
\end{definition}
Clearly, {$y^*\leq y_G^*$}.
Even more, we have the following.

\begin{theorem}\label{thm:inv}
If the original SDP is $G$-invariant then we have $y^{*}_G=y^*$.
\end{theorem} 
\begin{proof}
For every feasible $X$ and $g\in G$ the matrix $X^{g}$ is feasible. We have $\langle X,C\rangle=\langle X^{g},C\rangle$ for  every $g\in G$. Since the feasible region is convex we have
 \[{X_G:=\frac{1}{|G|}\sum_{g\in G} X^g}\] is feasible with  ${\langle X,C\rangle= \langle X_G,C\rangle}$ and $X_G^g=X_G$ for all $g\in G$. Thus {$y_G^{*}=y^{*}$}.
\end{proof}

Notice that the additional condition we impose in the above formulation, namely that $X=X^{g}$ for all $g\in G$ clearly reduces the dimension of the space of possible solutions and therefore the number of free variables in the formulation. This first step of a \emph{reduction to orbits} therefore reduces the intrinsic dimension of the problem. But in order to further simplify the formulation we also notice that Theorem \ref{thm:inv} allows us to restrict to invariant matrices (i.e., the commutant of the associated matrix representation). By Schur's Lemma we know that we can find a basis that \emph{block-diagonalises} the matrices in this space.
 Let \[\Sym_n(\R)=H_{1,1} \bot H_{1,2}\bot H_{1,m_i}\bot H_{2,1}\bot\cdots\bot H_{k,m_k}\] be an orthogonal decomposition into irreducibles, and pick an orthonormal  basis $e_{l,1,u}$ for each $H_{l,1}$.
 Choose $\phi_{l,i}:\, H_{l,1}\rightarrow H_{l,i}$ to obtain orthogonal bases \[e_{l,u,v}=\phi_{l,i}(e_{l,1,v})\] for each $H_{l,u}$. This then gives us an \emph{orthonormal symmetry adapted basis}.
 Now, for ever irreducible representation and every $(i,j)\in \{1,\ldots,n\}^2$ we can define \emph{zonal matrices} $E_l(i,j)$ with coefficients $$(E_{l}(i,j))_{u,v}:=\sum_{h=1}^{d_l}e_{l,u,h}(i)\cdot \overline{e^{k}_{l,v,h}(j)}.$$
 These characterise invariant $X\in \Sym_n(\R)$ via  $X_{i,j}=\sum_{l=1}^k\langle E_l(i,j),M_{l}\rangle$, for some $M_{l} \in \Sym_{m_l}(\R), 1\leq l\leq k$.
Summing up we have provided:
\begin{theorem}
A $G$-invariant SDP is equivalent to the following reduced SDP:
\[
  \begin{array}{rcl}
  \multicolumn{3}{l}{\inf \langle C,X\rangle} \\
  \text{s.t.} \quad \langle A_{i,j}, X \rangle & = & b_i \, , \quad 1 \le j, \le m \, ,1\leq i\leq k,  \\
  X_{i,j}&=&\sum_{l=1}^k\langle E_l(i,j),M_{l}\rangle,\\
    {M}_{l} & \succeq & 0, \text{ where } M_{l} \in \Sym_{m_l}(\R), 1\leq l\leq k.\\
  \end{array}
\]
\end{theorem}
\begin{trailer}{Some remarks}
The main work that has to be done to obtain this nice form (i.e., the actual calculation of the zonal matrices) is far from trivial. However the possible reductions can make a difference that enables one to actually compute a given SDP which otherwise might be too big to be handled by a SDP-solver. Furthermore,  in the context of semidefinite programming the issue with real irreducible representations versus  complex irreducible  representations can be easily avoided  by replacing symmetric matrices with hermitian matrices. This is why the definition of the zonal matrices was given with the complex conjugate, which is necessary only in the case of a complex symmetry adapted basis.
\end{trailer}
We want to highlight this potential reduction of complexity by studying the example of the $\vartheta$-number introduced in Definition \ref{def:theta} for graphs with symmetry. 

\begin{example}{Symmetry reduction for the $\vartheta$-number of a cyclic graph $C_n$}
    Consider a cyclic graph $C_n$ shown in the picture below for $n=10.$

    
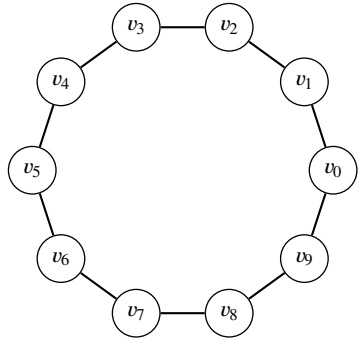
\begin{figure}[ht!]
    \begin{center}

 \begin{tikzpicture}
  \GraphInit[vstyle=Normal];
  \grCycle[prefix=v,Math=true,RA=2]{10};
\end{tikzpicture}
    \end{center}

\end{figure}
The $\vartheta$-number in this case is given  by
\begin{equation}\label{theta_cycel}
\begin{array}{lll} 
  \vartheta(C_n) & = \sup\big\{  \sum_{i,j} B_{i,j} : &
  B\in \Sym_n(\R),\ B \succeq 0 \\
&&   \sum_{i=1}^n B_{i,i}=1,\\
&&   B_{i,j}=0\  \quad  j\equiv i+1 \quad (\textrm{mod}\ n)\big\}.\\
\end{array}
\end{equation}
We see that this SDP is invariant under the natural action of the cyclic group $\Z/ n\Z$. 
We get  a symmetry adapted basis for the (complexification) of this representation through the Fourier basis, and a corresponding real decomposition. Then one obtains the following equivalent formulation for the SDP \eqref{theta_cycel} 
\begin{equation}\label{theta primal2}
\begin{array}{lll} 
  \vartheta(C_n) & = \sup\big\{  n\cdot x_0 : &
  (x_0,\ldots,x_{\lfloor \frac{n}{2}\rfloor})\in\R_{\geq 0}\ \\
&&   \sum\limits_{j=0}^{\lfloor \frac{n}{2}\rfloor} x_j=1,\\
&&   \sum\limits_{j=0}^{\lfloor \frac{n}{2}\rfloor} x_j\cdot\cos(\frac{2j\pi}{n})=0\big\}\\
\end{array}
\end{equation}
Through this formulation of  the $\vartheta$-number as a linear program, it is possible to calculate the  $\vartheta$-number directly. Indeed, we can deduce \cite[Corollary 5]{lovasz79}
$$\vartheta(C_k) =
\begin{cases}
  \frac{n \cos(\pi/k)}{1 + \cos(\pi/k)} & \text{for odd }  k, \\
  \frac{k}{2}                           & \text{for even } k.
\end{cases}$$
\end{example}
Recall that our main objective is polynomial optimization. We now turn our attention to specific SDPs stemming from sums of squares approximations.
Note that in this situation the action of a group $G$  on $\R^n$ also induces an action on  the $\R-$vector space $\R[X]$, as we have seen before and a matrix representation $M(G)$ for the space of polynomials of degree at most $d$. 
Thus, if $p^g=p$ for all $g\in G$ we can define

\[Q^G:=\frac{1}{|G|}\sum_{g\in G}M(g)^TQM(g),\]
and will have \[p \ = \ (Y)^T Q^G Y\] with the property that now $Q^G$ commutes with the matrix representation $M(G)$. Therefore the above methods for general SDPs can be used again to block-diagonalize the matrix $Q^G$ and thus simplify the calculations. This was first explored in detail by Gatermann and Parillo \cite{gatermann2004symmetry} and several other authors \cite{cimprivc2009sums,riener2013exploiting,dostert2017new}. 
We want to highlight this in the following example.

\begin{example}{Symmetry reduction of the SOS decomposition of a symmetric quadratic}
Consider the homogeneous polynomial \[p=a\cdot\sum_{i=1}^n X_i^2+b\cdot \sum_{i<j} X_{i}X_{j}.\] We want to examine the conditions on $a$ and $b$ such that $p$ is a sum of squares. Since $p$ is homogeneous and of degree 2, a sum of squares decomposition will involve only squares of homogeneous polynomials of degree 1. 
In other words, by Proposition~\ref{prop:sossemidefinite} we consider the vector $Y=(X_1,\ldots,X_n)$ comprised of all monomials of degree 1 and have that  the polynomial $p$ is a sum of squares if and only if we can find a positive semidefinite matrix $Q$ of size $n\times n$ such that \[p=Y^TQY.\] 
Now, by construction, $p$ remains invariant by the permutation action of $\S_n$ on the monomials $X_1,\ldots, X_n$. This action is represented by the \emph{permutation matrices}, and we can therefore assume that $Q$ commutes with every permutation matrix. As we have seen above, this representation decomposes into two non-isomorphic irreducible representations, one isomorphic to the trivial representation and the other one to the Specht module $W^{(n-1,1)}$.
Therefore, in a symmetry adapted basis, we may assume that the matrix $Q$ is of the from 
$$Q=\begin{pmatrix}
\alpha&0&0&\ldots&0\\
0&\beta&0&\ldots&0\\
\vdots&&\ddots&\\
0&0&&\ldots&\beta\\
\end{pmatrix}.
$$
Doing the calculations we obtain that a sums of squares decomposition of $p$ is of the form \[p=\alpha(X_1+X_2+\ldots+X_n)^2+\beta\sum_{i<j}(X_i-X_j)^2=(\alpha+(n-1)\beta)\sum_{i=1}^n X_i^2+ 2(\alpha-\beta)\sum_{i<j} X_iX_j,\]
with $\alpha,\beta\geq 0.$
\end{example}
In the above example we have seen that the sum of squares decomposition essentially consisted of two type of summands. Firstly the square of the invariant polynomial $X_1+X_2,+\ldots+X_n$ and summands which are squares of elements in the $\S_n$-orbit of  $X_1-X_2$. This observation can be made more precise with the following setup: if  we want to decide for a given polynomial $f$ of degree $2d$ which is invariant by a group $G$ if it is a sum of squares we consider  the vector space $\R[X]_{\leq d}$ of polynomials of degree at most $d$, which is  finite dimensional. Then the linear action of a group $G$ on $\R^n$ also induces  a linear action on $\R[X]_{\leq d}$ and  we  can consider the isotypic decomposition 
\begin{equation}\label{eq:decomp1}
\begin{mathbf} \R[X]_{\leq d}=\bigoplus_{i\in {I}}V_i =\bigoplus_{i \in {I}} \bigoplus_{j=1}^{\eta_i} W_{i,j}\end{mathbf}
\end{equation}

where the $W_i$ are again the non-isomorphic irreducible representations of $G$, and any $W_{i,j}$ is an irreducible representation isomorphic to $W_i$.
For the discussion we now assume for simplicity that every irreducible component $W_i$ of this decomposition is real irreducible. In this case we can construct a set of real polynomials  \[\{f_{11},f_{12},\ldots,f_{l \eta_l}\}\subset\R[X]_{\leq d}\] with the following two properties: \begin{enumerate} \item The polynomial $f_{i1}$  cyclically generates $W_{i1}$ i.e., for every fixed $i$ the $G$-span \[ \langle f_{i1} \rangle_G= W_{i,1}.\]
\item The polynomial $f_{ij} \in W_{i,j}$ is the image of $f_{i1}$ under the (up to scalar multiplication) unique  $G$-isomorphism between $W_{i,1}$ to $W_{i,j}$. This in particular implies that   $f_{ij}$ generates cyclically $W_{i,j}.$\end{enumerate}
We denote by 
$\left( \langle f_{i1},\ldots,f_{i\eta_i} \rangle_\R^2 \right)$ a sum of squares of elements in the vector space spanned by the polynomials $f_{i1},\ldots,f_{i\eta_i}$. With this notation, and integrating the described construction  with the consequences of Schur's Lemma (Theorem \ref{thm:schur}),   one arrives at the following  which agrees with the observation from the example above (see also \cite{riener2,cimprivc2009sums,gatermann2004symmetry,riener2013exploiting,debus2020reflection} for more details on  the following statement). 
\begin{theorem}\label{THM Decomp}
With the notation defined above  any $G$-invariant polynomial $p$ of degree $2d$ that is a sum of squares  can be written in the from\begin{align*}
   p=\sum_{i\in I} \sum_{j=1}^{\eta_i}\sum_{g\in G}p_{ij}^g 
    \end{align*}
    where every $p_{ij}\in \left( \langle f_{i1},\ldots,f_{i\eta_i} \rangle_\R^2 \right)$.
\end{theorem}
\begin{proof} 
Let  $p$ be a  sum of squares. Then there is symmetric positive semidefinite bilinear from   $$B: \R[X]_{\leq d}\times \R[X]_{\leq d}\rightarrow \R$$ which is a Gram matrix for $p$, i.e.  for every $x\in\R^{n}$ we can write \[p(x)=B(Y^{d},Y^{d}),\] where $Y^{d}$ stands for the vector of monomials up to degree $d$. 
Since $p$ is $G$-invariant,  by linearity, we may assume that  $B$ is a $G$-invariant bilinear form. 
We get from Corollary~\ref{cor:decomp}  
that $B^{ij}(v,w)=0$ for all $v\in V_{i}$ and $w\in V_{j}$, i.e., the isotypic components appearing in \eqref{eq:decomp1} are orthogonal with respect to $B$ and hence it suffices to look at $$B^{jj}: V_{j}\times V_{j}\rightarrow \R$$ individually. 
If $B^{jj}=0$, there is nothing to prove. Otherwise, consider the decomposition \[V_{j}:=\bigoplus_{k=1}^{\eta_j} W_{j,k},\] where with the notation above each  $W_{j,k}$ is generated by the orbit of a polynomial  $f_{j,k}$. We can freely  identify $V_j$ with its complexification $V_j^{\C}$ since by assumption all irreducible representations are irreducible over $\C$. Furthermore, since $B$ is positive semidefinite and $f_{j1}$ cylically generates $W_{j1}$, we may assume that $B(f_{j1},f_{j1})>0$. We extend $f_{j1}$ to a basis   \[f_{j1}=f_{j11},f_{j12},\cdots,f_{j1d}\] of $W_{j,1}$, and denote by  $f_{ji1},f_{ji2},\cdots,f_{jid}$  the image of this basis under the unique $G$-isomorphism sending  $f_{j1}$ to $f_{ji}$. Consider the vector space $U$ generated by $f_{j1},\cdots,f_{j\eta_j}$. Since the restriction $\restr{B}{U\times U}$  is a positive semi-definite bilinear form, we obtain \[\sum_{a=1}^{\eta_j}\sum_{b=1}^{\eta_j} B(f_{ja},f_{jb})f_{ja}f_{jb} = \sum_{t=1}^z g_t^2\] for some $g_t \in \langle f_{j1},\cdots,f_{j\eta_j}\rangle_\R$.
Now, clearly  there exists a symmetric, $G$-invariant, bilinear form $C : V_j \times V_j \rightarrow \R$ for which we have \[\mathcal{R}_G(\sum_{i=1}^z g_t^2) = C(Z,Z),\] where \[Z=(f_{j11},\ldots,f_{j1d},f_{j21},\ldots,f_{j\eta_j1},\ldots,f_{j\eta_jd}).\]
To conclude we need to see that the form  $C$ is essentially $B^{jj}$, i.e., that it is obtained from $B^{jj}$ by multiplication with a positive scalar. This can be done by direct calculations using the fact that each $f_{jab}^g$ for $g \in G$  can be expressed in the basis of $W_{j,a}$ defined above. It then follows that there exists a  $\lambda_{jbb'}$, independent on $a,a'$ such that \[C(f_{jab},f_{ja'b'}) = \lambda_{jbb'}C(f_{jab},f_{ja'b'}).\] Since $B$ and $C$ can be seen as elements of $\Hom_G(W_{j1},W_{j1}^*)$, applying Schur's lemma  one can prove that $\lambda_{jbb'}$ is actually independent of $b,b'$, that is $C=\lambda B^{jj}$ for some constant $\lambda \in \R^*_+$.
\end{proof}
In analogy to the zonal matrices defined before, we can also reformulate Theorem \ref{THM Decomp} in terms of matrix polynomials, i.e., matrices with polynomial entries. To do this we define a block-diagonal symmetric matrix $B$ with $j$ blocks $B^{(1)},\dots,B^{(j)}$ with the entries of each block given by:
\begin{equation}\label{eq:B}
B^{(j)} = \left( \sum_{g\in G} (f_{ju}^g \cdot f_{jv}^g )\right)_{u,v} .
\end{equation}
With these notations  Theorem \ref{THM Decomp} can be stated in the  following equivalent form.
\begin{corollary}\label{cor:sos}
Let $g \in \R[X]_{\leq 2d}^G$. Then $g$ is a sum of squares if and only if \begin{align*}
    g = \langle A_1 \cdot B^{(1)} \rangle + \ldots +  \langle A_l \cdot B^{(l)} \rangle,
\end{align*}
for some symmetric and positive semidefinite matrices $A_j \in \emph{Mat}_{\eta_j \times \eta_j}(\R)$.
\end{corollary}     
\begin{remark}
The restriction to type I real irreducible representations in the description above is mainly due to convenience of the presentation. Using the discussion in in Section \ref{seq:rep} on real irreducible representations    it is possible  to generalize to type II and  to adapt to type III representations.  
\end{remark}                                                                                                                      
These first examples give a hint on the power of symmetry reduction to simplify optimization problems affording some symmetries. This approach has been used quite successfully in a wide range of applications and settings.
In the next section, we focus on the case of sums of squares and explore how the additional algebraic structure of polynomials can be combined with the above approach to get a better understanding of invariant sums of squares. 

\section{Invariant theory}\label{sec:invariant}

Polynomial functions that remain unchanged under the action of a symmetry transformation naturally appeared in the previous section. 
Such functions are called \emph{polynomial invariants}, and they play a crucial role in \emph{invariant theory}, a branch of algebra that studies symmetry in algebraic structures, such as groups, rings, and fields.
The structural insights into the set of invariant polynomials can give important additional comprehension and different ways of simplifications for polynomial optimization, as we will outline in this section. In order to do this, we give first an overview of the basic concepts and important results in invariant theory which will prove useful in order to take advantage of symmetries in polynomial optimization problems. For details and further insights we refer the reader to \cite{sturmfels,derksen2015computational,inv1,inv2,inv3}.
\subsection{Basics of invariant theory}\label{ssec:BasInv}
We begin with examining the situation of a linear action of a group on the ring of polynomials more closely. As was introduced before, given a finite group $G$  and a representation  $\rho: G\longrightarrow \K^n$ we can define a linear action $f^g$ on the polynomial functions on $V$, denoted $\K[V]$ via $f^g:=f\circ \rho(g^{-1})$. Choosing a basis of $V$ we can identify $\K[V]$ with $\K[X_1,\ldots,X_n]$  and consider the associated isotypic decomposition of the $G$ module:
\begin{equation}\label{decom1}
\K[X_1,\ldots,X_n]=\bigoplus_{i\in I} \K[X_1,\ldots,X_n]^{\chi_i}, 
\end{equation}
where as before $I$ is the set of isomorphism classes of irreducible representations and $\chi_i$ for $i\in I$ is a list of irreducible characters. In addition to being a $\K$-vector space with group action, $\K[X_1,\ldots, X_n]$ is in fact an algebra and this additional algebraic structure also allows for a finer analysis of the group action.  If $\chi$ corresponds to the trivial character the polynomials in $\K[X_1,\ldots,X_n]^\chi$ will not be affected by the action of $G$. Polynomials with this property, i.e, elements \[f\in\K[X_1,\ldots, X_n]\text{ such that } f^g=f \text{ for all }g\in G\]  are called \emph{invariant polynomials}.  Since the  property of being invariant  is not affected by addition or multiplication with other invariant polynomials the set of invariant polynomials forms a sub-algebra denoted by  $\K[X]^G$. The polynomials belonging to the other isotypic components are sometimes also called \emph{semi-invariants}.

\begin{example}{The Motzkin polynomial as an example for a symmetric polynomial}
We consider the group $\S_2$ acting linearly on $\C^2$ by permuting coordinates. Then, the \emph{Motzkin polynomial}
\[M:=X_1^4X_2^2+X_1^2X_2^4-3X_1^2X_2^2+1\] is invariant with respect to this action. In the case of $\S_n$ invariant polynomials we also speak of \emph{symmetric polynomials}.
\end{example}
A very useful tool to work with polynomials in invariant setting is the so-called \emph{Reynolds operator} defined as follows. 
\begin{definition}
For a finite group $G$ the map \[\mathcal{R}_G(f):=\frac{1}{|G|}\sum_{g\in G}f^g\] is called the \emph{Reynolds operator} of $G$.
\end{definition}
The Reynolds operator of a finite group $G$ has the following properties:
\begin{enumerate}
\item $\mathcal{R}_G$ is a $\K[X]^G$- linear map.
\item For $f\in \K[X]$ we have $\mathcal{R}_G(f)\in\K[X]^G$.
\item $\mathcal{R}_G$ is the identity map on $\K[X]^G$, i.e $\mathcal{R}_G(f)=f$ for all $f\in K[X]^G$.
\end{enumerate}

\begin{remark}
Whereas we introduced the Reynolds operator for finite groups, it can be more generally defined for compact and reductive groups and most of the results we state for finite groups mostly can be adapted for this more general case.
We refer the reader to \cite{sturmfels, derksen2015computational}   for more details and algorithmic questions.
\end{remark}
As seen above, $\K[X]^G$ is a subalgebra of $\K[X]$. In the second half of the 19th century, invariant theory was very much concerned with the following question:  is this subalgebra generated by finitely many elements? The following Theorem, proven by Hilbert in 1890, initiated modern invariant theory. 
\begin{theorem}[Hilbert]\label{thm:hilbertfinite}
Let $G$ be a finite group. Then the invariant subalgebra $\K[X]^G$  is generated by finitely many homogeneous invariants, i.e., there is a finite set of invariant homogeneous polynomials $\pi_1,\ldots,\pi_m$ such that every invariant polynomial $f\in \K[X]^G$ can be written as a polynomial in $\pi_1,\ldots,\pi_m$.
\end{theorem}
\begin{remark}
Hilbert himself asked in his 14th problem whether the above Theorem holds generally for all groups. Indeed, it holds for a large class of infinite groups. However, with a concrete counter example,  Nagata~\cite{nagata195914} could prove in 1959 that not all invariant algebras are finitely generated. 
\end{remark}
In the case of the symmetric group $\S_n$ the following families of generators are well known and used in many different parts of algebra and combinatorics.
\begin{definition}
For $n\in \N$, consider the following two families of symmetric polynomials.
\begin{enumerate}
 \item For $0\leq k\leq n$ let $p_k:=\sum_{i=i}^{k}x_i^k$ denote the $k$-th power sum polynomial,
 \item For $0\leq k\leq n$ let $e_k:=\sum_{1\leq i_1<i_2<\ldots <i_k\leq n} x_{i_1}x_{i_2}\cdots x_{i_k}$ denote the $k$-th elementary symmetric polynomial.
\end{enumerate}
\end{definition}
\begin{theorem}
The two families $\{p_1,\ldots, p_n\}$ and $\{e_1,\ldots, e_n\}$  both are algebraically independent generating sets for the algebra of $\S_n$-invariant polynomials, namely
  \[\C[X_1,\ldots,X_n]^{\S_n}=\C[e_1,\ldots, e_n]=\C[p_1,\ldots,p_n].\]
\end{theorem}

\begin{example}{The Motzkin polynomial expressed  in elementary symmetric polynomials}
Consider the action of $\S_2$ on $\C^2$. Then, every symmetric polynomial can be uniquely written as a polynomial in $e_1:=X_1+X_2$, and $e_2:=X_1X_2$, and the Motzkin polynomial $M$ from the example above can be expressed as 
\[M=e_1^2e_2^2 - 2e_2^3 - 3e_2^2 + 1.\]
\end{example}
\begin{remark}
Since both the power sum polynomials and the elementary symmetric polynomials generate the ring of symmetric polynomials, each of two families of symmetric polynomials can be expressed by the other. The expression can be deduced from the so-called Newton identities (see e.g.\cite{M}):
\begin{equation}\label{eq:newton}
 k(-1)^{k} e_k(x)+\sum_{i=1}^{k}(-1)^{i+k}p_{i}(x)e_{k-i}(x)=0.
\end{equation}
\end{remark}
In contrast to the above example of the symmetric polynomials, for general groups  we might need more than $n$ generators. We can, however, give a bound for the number of generators needed.
\begin{theorem}[Noether's bound]\label{thm:Noether}
Let $G$ be a finite group  acting linearly on $\K^n$. Then $\K[X]^G$ is generated as an algebra over $\K$ by not more than $\binom{n+|G|}{n}$ many homogeneous invariants of  degree not exceeding $|G|$. 
\end{theorem}
So, in particular, it is not guaranteed that a set of  algebraically independent generators exists. However, more generally, it can be shown that the
invariant ring is Cohen-Macaulay and that a Hironaka decomposition exists, i.e., that we can find $n$ algebraically independent polynomials $\theta_1,\ldots,\theta_n$ and polynomials $\eta_1,\ldots,\eta_k$ such that 
\[
\K[X]^G =\bigoplus_{i=1}^k\eta_i\K[\theta_1,\ldots,\theta_n].
\]
The algebraic invariant polynomials  $\theta_1,\ldots,\theta_n$ are called \emph{primary invariants}, and the $\eta_1,\ldots,\eta_k$ are called \emph{secondary invariants}. With these families of polynomials, every invariant polynomial can be uniquely written in such a way that the secondary invariants appear only linearly. 
\begin{example}{Invariant polynomials for the alternating group}
Consider the group $\mathfrak{A}_3$ which is the subgroup of $\S_3$ of even permutations and define $\Delta=(X_1-X_2)(X_1-X_3)(X_2-X_3)$ Then for every $\mathfrak{A}_3$ invariant polynomial $f$ there exist two unique symmetric polynomials $g_0(p_1,p_2,p_3)$ and $g_1(p_1,p_2,p_3)$ such that 
\[
f(X_1,X_2,X_3)=g_0(p_1,p_2,p_3)+g_1(p_1,p_2,p_3)\cdot \Delta.
\]
In other words, the  $p_1,p_2,p_3$ (or any other generating set for the $\S_3$-invariant polynomials) can be used as primary invariants and $\Delta$ serves as  the secondary invariant.
\end{example}
As we have seen now, the trivial component of the isotypic decomposition \eqref{eq:decomp1} enjoys the property of being a finitely generated $\K$-algebra. We now turn to the other isotypic components. Let $\chi$ be any character of $G$ and $f\in\K[X]^G$ be an invariant polynomial. Then clearly the application \[\abb{\cdot f }{\K[X]^\chi}{\K[X]^\chi}{q}{q\cdot f}\] which corresponds to multiplication by $f$ is a $G$-homomorphism. So $\K[X]^\chi$ has in turn the structure of a $\K[X]^G$-module. We have in fact the following.

\begin{theorem}
Let $\chi$ be an irreducible character of a finite group $G$. Then $\K[X]^\chi$ is a  is a finitely generated (even Cohen-Macaulay) $\K[X]^G$-module and $\K[X]^\chi$ is generated by homogeneous elements of degree not exceeding $|G|$. In particular, also $\K[X]$ is a finitely generated $\K[X]^G$ (Cohen-Macaulay) module. 
\end{theorem}

\begin{example}{Decomposition of polynomials in two variables}
We consider the case of polynomials in two variables $\K[X_1,X_2]$ and their relation to $\S_2$ symmetric polynomials. Here we have that $\K[X_1,X_2]$ is generated as an $\K[X_1,X_2]^{\S_2}$-module by 1 and $X_1-X_2$, i.e., every polynomial $f$ can be written as 
\[f=g_0+g_1\cdot (X_1-X_2),\] where $g_1,g_2\in\K[X_1,X_2]^{\S_2}$. Moreover, in this case the representation is in fact unique.  
\end{example}

\begin{definition}
A \emph{reflection} is a mapping from a Euclidean space to itself that is an isometry whose set of fixed points is a hyperplane. If $V$ is an Euclicean vector space we can view a reflection as an  element in $\OO(V)$ that has exactly one  eigenvalue different from $1$, i.e., one other eigenvalue that is $-1$. If $(V,\rho)$ is a representation of $G$, then $g \in G$ is a reflection if the mapping $v \mapsto \rho(g)v$ is a reflection and we say that $G$ is a \emph{reflection group} if it is generated by reflections.
\end{definition}
It is very important to remember that the property of being a reflection group is not a property of the group itself, but it is crucially linked to the representation. This can be seen in the following example.
\begin{example}{Two different representations of $\S_n$}\label{exa: reflection group}
We consider the symmetric group $\S_n$. This group is generated by the transpositions $\tau_i=(i,\ i+1)$ for $i \in \{1,\cdots,n-1\}$. 
\begin{enumerate}
\item Consider the defining  representation of $\S_n$ given by permuting the coordinates in $\R^n$. In this representation,  elements in $O(V)$ corresponding to transpositions have the  eigenvectors $e_j$ for $j \neq i,i+1$ with eigenvalue $1$, as well as $e_i+e_{i+1}$ with eigenvalue $1$ and $e_i-e_{i+1}$ with eigenvalue $-1$. Therefore these elements are reflections, which geometrically correspond to the reflection through the hyperplane of equation $x_i-x_{i+1}=0$ in $V$. With this representation $\S_n$ is acting as a reflection group.
\item Now, let $n=2$ and consider the representation on \[
V=\R\cdot e_1 \oplus \R \cdot e_2 \oplus \R \cdot e_3\; \text{    given by }\] 
\[(1,2) \cdot e_1=e_2, (1,2)\cdot e_2=e_1 \text{ and } (1,2)\cdot e_3 = -e_3. 
\]
Then, the image of the transposition $(1,2)$ has eigenvalues $1,-1,-1$ and thus is not a reflection. Therefore, $\S_2$ is not a reflection group if acting in this way.  
\end{enumerate}
\end{example}
The following theorem is due to Shephard, Todd, Chevalley and motivates the study of reflection groups.

\begin{theorem}[Shephard, Todd, Chevalley]\label{thm:STC}Let $G$ be a group with a finite dimensional representation $V$. Then the following properties are equivalent:\begin{enumerate}
    \item $G$ is a reflection group;
    \item the corresponding invariant ring $\K[V]^G$ is a polynomial algebra;
    \item the polynomial ring $\K[V]$ is a free $\K[V]^G$ module.
\end{enumerate}
\end{theorem}
In the case of a  finite reflection group $G$ the collection $\pi_1,\ldots,\pi_n$ which generate the invariant ring  are called \emph{basic invariants}. These basic invariants are not unique, but their sequence of degrees $d_i(G) := \deg \pi_i$ is unique. The following example highlights the equivalence of the statements in Theorem \ref{thm:STC}.
\begin{example}{Invariants of a different $\S_2$ action}
We consider the group $\S_2$ with the action on $\K^3$ above such that $\S_2$ is not acting as a reflection group. Then, \[\K[x_1,x_2,x_3]^G=\K[\tau_1,\tau_2,\tau_3,\tau_4]\] with $\tau_1=x_1+x_2$, $\tau_2=x_1x_2$, $\tau_3=x_3^2$ and $\tau_4=x_3(x_1-x_2)$. Clearly, these three polynomials cannot be algebraically independent and indeed, there is an algebraic dependency between the $\tau_i$'s, namely \[\tau_1^2\tau_3 - 4\tau_2\tau_3 - \tau_4^2=0.\] Similarly one can find that  $\K[x_1,x_2,x_3]$ is not a free $\K[\tau_1,\tau_2,\tau_3,\tau_4]$-module.
\end{example}

\begin{definition}
Let $G$ be a finite reflection group. Then, the  quotient $\K$-algebra of the polynomial ring modulo the ideal generated by the non-constant elements of the invariant ring is called the covariant algebra of $G$ and denoted by $\K[X_1, \ldots, X_n]_G$. If $\pi_1,\ldots,\pi_n$ are a minimal set of algebraically independent generators of the invariant ring, we have $$\K[X_1, \ldots, X_n]_G := \K[X_1, \ldots, X_n] / \left(\pi_1,\ldots,\pi_n \right)_{\K[X_1,\ldots,X_n]}.$$

\end{definition}
The covariant algebra of $G$ has the structure of a $G$-module and we find in fact the following. 
\begin{theorem} \label{thm:covariant algebra}
Let $G$ be a real reflection group acting linearly on $\R^n$. Then the covariant algebra $\R[X_1,\ldots,X_n]_G$ is as $G$-module isomorphic to the regular representation $\R[G]$ and $$\R[X_1,\ldots,X_n] \cong \R[X_1,\ldots,X_n]^G \otimes_\R \R[X_1,\ldots,X_n]_G$$ as graded $\R$-algebras.  
\end{theorem}
In fact, for general finite groups we have the following analog which follows from the Hironaka decomposition. 
\begin{theorem}
Let $G$ be a finite group with a linear action on a finite dimensional vector space $V$. Let $\theta_1,\ldots,\theta_n$ be the primary invariants and suppose that the rank of the $\K[\theta_1,\ldots,\theta_n]$-module $\K[V]^G$ is $m$. Then $\K[V]/(\theta_1,\ldots\theta_n)$ is isomorphic as $G$-module to $m$ times the regular representation. 
\end{theorem}
An algorithmic approach to efficiently compute the decomposition above can for example be found in \cite{hubert2022algorithms}.
\begin{trailer}{Going via $G$- harmonic polynomials}
Computing a basis of the covariant algebra may be complicated and involve  calculation of a Groebner basis. We shortly mention that also the set of $G-$harmonic polynomials defined below can be  efficiently used.  
\end{trailer}
\begin{definition}
For a polynomial $f=\sum_\alpha c_\alpha X_1,\ldots,X_n^\alpha  \in \R[X_1,\ldots,X_n]$ we denote by $f(\partial)$ the linear operator 
\[ \abb{f(\partial )}{\R[X_1,\ldots,X_n]}{\R[X_1,\ldots,X_n]}{g}{\sum_\alpha c_\alpha \frac{\partial^{|\alpha|}}{ \partial X_1^{\alpha_1}\ldots\partial X_n^{\alpha_n}}g,}\] 
i.e., $f(\partial)$ is the formal sum of scaled partial derivatives considered as a linear map. 
\end{definition}
{\begin{example}{Example for $f(\partial)$}
Let $f = X_1^2+X_1X_2 \in \R[X_1,X_2,X_3]$, then $f(\partial) = \frac{\partial^2}{\partial X_1 \partial X_1} + \frac{\partial^2}{\partial X_1 \partial X_2}$ and $$f(\partial) \left( X_1^2+X_2^2+X_3^2 +X_1X_2X_3 \right) = 1 +X_3.$$
\end{example}}

\begin{definition}
Let $G$ be a reflection group with invariant ring  $\R[\pi_1,\ldots,\pi_n]$. Then the $\R$-vector space of harmonic polynomials is defined as \[\mathcal{H}_G := \left( \R[X_1,\ldots,X_n]^G\right)^\perp,\] with respect to the scalar product on $\R[X_1,\ldots,X_n]$ given by $$ \abb{\langle \cdot , \cdot \rangle}{\R[X_1,\ldots,X_n] \times \R[X_1,\ldots,X_n]}{\R[X_1,\ldots,X_n]}{(f,g)}{ \ev_{(0,\ldots,0)}\left( f(\partial) g(X_1,\ldots,X_n) \right).}$$
\end{definition}
Now the following statement shows that $G$-harmonic polynomials have some remarkable similarities with the covariant algebra as presented in Theorem \ref{thm:covariant algebra}. We refer the reader also to \cite{bergeron2009algebraic} for details on the following Theorem.
\begin{theorem}
Let $G$ be a real reflection group and $\Delta:=\prod L_i$, be the product of the linear polynomials defining the reflection hyperplanes. 
\begin{enumerate}
\item The vector space of $G$-harmonic polynomials $\mathcal{H}_G$ is generated by all partial derivatives of $\Delta$, i.e., $\mathcal{H}_G = \langle \frac{\partial^\alpha}{\partial x^\alpha}\Delta : \alpha \in \N_0^n \rangle_\R$. 
\item Furthermore, $\mathcal{H}_G$ is as $G$-representation isomorphic to the regular representation of $G$ and $\R[X_1,\ldots,X_n] = \R[\pi_1,\ldots,\pi_n]\otimes_\R \mathcal{H}_G.$ 
\end{enumerate}
\end{theorem}

Let $G$ be a finite reflection group and $\psi_1,\ldots,\psi_n$ be generators of the invariant ring.  Consider the map $$\abb{\Psi}{\R^n}{\R^n}{X_1,\ldots,X_n}{(\psi_1(X_1,\ldots,X_n),\ldots,\psi_n(X_1,\ldots,X_n)).}$$
Then, thanks to a statement of Steinberg in \cite{steinberg1960invariants}, we have $$\Delta = c\cdot \emph{Jac } \Psi,$$ where $c \in \R \setminus \{0\}$ and $\emph{Jac } \Psi$ denotes the Jacobian matrix of $\Psi$. The choice of fundamental invariants $\psi_1,\ldots,\psi_n$ does not matter.

\begin{example}{Jacobian of $\S_n$}
For $\S_n$ the symmetric group acting on $\R^n$ via coordinate permutation and \[\psi_i := p_i = \sum_{j=1}^n X_j^i\] the power sums, we obtain $\Delta = \prod_{i < j}(x_i-x_j)$ equals (up to a scalar) the determinant of the Vandermonde matrix. The form defined above $\Delta$ which is the product over all reflections $(i,j)$ of $\S_n$ equals therefore up to a scalar the Jacobian of $\Psi$. 
\end{example}

Building on the notions of invariant theory outlined here we now can outline how to use the observation that the polynomial ring is a module over the invariant ring in the context of sums of squares formulations.

\subsection{Invariant theory and sums of squares}\label{ssec:InvSOS}

In Section \ref{seq:rep} where we focused on  representation theory, we already saw how to apply Schur's Lemma to obtain simplifications for sums of squares computations. 
We shall explain how to combine the techniques developed with Schur's Lemma with results from invariant theory in the situation of symmetric sums of squares. Note that an invariant polynomial which can be expressed as a sum of squares in the ring $\R[X_1,\ldots,X_n]$ will not necessarily have a sum of squares decomposition in invariant polynomials, i.e., 
$$\R[X_1,\ldots,X_n]^G\bigcap\sum \R[X_1,\ldots,X_n]^2\neq \sum(\R[X_1,\ldots,X_n]^G)^2,$$ as can be seen in the following example.
\begin{example}{A symmetric sum of squares is not a sum of symmetric squares}
Consider the symmetric sums of squares polynomial $f=X_1^2+X_2^2+\ldots+X_n^2$.  Up to scalar multiplication there is only one symmetric polynomial of degree 1, namely $e_1:=X_1+X_2+...+X_n$.
Clearly, if $n>1$ we have $\alpha\cdot e_1^2\neq f$ for every $\alpha\in\R$.
\end{example}
The above example suggests that it is at first sight not clear which algebraic property in the ring $\R[X]^G$ certifies that a given invariant polynomial is a sum of squares of elements in $\R[X]$. However,  we have seen that $\R[X]$ is a finitely generated module over the invariant ring. This observation allows for a refined version of Theorem \ref{THM Decomp}.  Following the presentation in \cite{debus2020reflection}, we focus here mainly on finite reflection groups, since, as we have seen, their representation theory is particularly nice. However,  these results can be generalized to all finite groups using the Hironaka decomposition with appropriate adaptations. We begin with the following observation which makes use of the decomposition
\[\R[X_1,\ldots,X_n]\simeq \R[X_1,\ldots,X_n]^G \otimes \R[X_1,\ldots,X_n]_G\] provided in Theorem \ref{thm:covariant algebra}. Since the covariant algebra is isomorphic to the regular representation of $G$, we can pick a basis 
$S:=\{s_1,\ldots,s_{|G|}\}$. Relatively to this basis  we now construct  a matrix polynomial 
with entries 
$$H^S_{u,v}:=\mathcal{R}_G(s_u\cdot s_v),\, \text{ where }1\leq u,v\leq |G|.$$
Since now every entry $\mathcal{R}_G(s_u\cdot s_v)$ is by construction $G$-invariant we can express the entries in  terms of the $\pi_1,\ldots,\pi_n$, i.e., we obtain a matrix polynomial in new variables $z_1,\ldots,z_n$
\[H^S(z_1,\ldots,z_n)\in\R[z_1,\ldots,z_n]^{|G| \times |G|}.\]
Given any matrix-polynomial $L(z_1,\ldots,z_n)\in\R[z_1,\ldots,z_n]^{n\times m}$,  
we can construct a square matrix polynomial  $M=L^TL$. We say that a $n\times n$ matrix polynomial $M$ is a \emph{sum of squares matrix polynomial} if it can be obtained in this way. With this notion at hand one  can  deduce from the decomposition in Theorem \ref{thm:covariant algebra} the following  algebraic certificate  for an invariant polynomial to be a sum of squares.
\begin{proposition}
Let $G$ be a finite reflection group and let $f$ be an invariant polynomial.  Consider the polynomial  $g\in\R[z_1,\ldots,z_n]$ with $g(\pi_1,\ldots,\pi_n)=f$. Then $f$ is a sum of squares if and only if $g$ admits a representation of the form
$$g= \Tr (M\cdot H^S),$$
where $M$ is a sum of squares matrix polynomial.
\end{proposition}
The above construction works for every basis of the covariant algebra. But since we know that the group representation  on the covariant algebra is the regular representation one can pick a basis that decomposes this covariant algebra into irreducible representations. We note this explicitly in the following Proposition.

\begin{proposition}\label{cor:deomposition of polynomial ring}
Let $G$ be a finite reflection group, and  let \[\R[X_1,\ldots,X_n]_G = \bigoplus_{i\in I} \eta_i \theta^{(i)}\] be the isotypic decomposition of the covariant algebra. Denote $\ell=|I|$. Then there are polynomials $s_{11},\ldots,s_{\ell\eta_\ell} \in \R[X_1,\ldots,X_n]_G$ such that any $f \in \R[X_1,\ldots,X_n]$ can be written as $$f = \sum_{i\in I} \sum_{j=1}^{\eta_i} \sum_{\sigma \in G} g_{ij,\sigma} \sigma s_{ij}, $$ where $g_{ij,\sigma} \in \R[X_1,\ldots,X_n]^G$.
\end{proposition}
Note that the summation index $\eta_i$ equals the multiplicity of the corresponding   irreducible representation in the isotypic decomposition, which in turn equals the dimension of the irreducible representation. 
\begin{definition}
With the notation used above, we can construct 
a matrix polynomial $H^{\theta_i}\in\R[z_1,\ldots,z_n]^{\eta_i\times\eta_i}$ for  every irreducible representation $\theta^{(i)}$ of $G$ via 
$$H^{\theta_i}_{u,v}=\mathcal{R}_G(s_{i,u}\cdot s_{i,v}), \text{ where } 1\leq u,v\leq \eta_i.$$
\end{definition}
Combining the above definition with the results from Schur's lemma we immediately get the following.
\begin{theorem}\label{thm:sums of squares description}
Let $G$ be a finite reflection group with $\R[X_1,\ldots,X_n]^G \simeq\R[\pi_1,\ldots,\pi_n]$, then we have
$$\Sigma \R[X_1,\ldots,X_n]^2 \cap \R[X_1,\ldots,X_n]^G= \left\{ g\in \R[\pi_1,\ldots,\pi_n]\,:\, g=\sum_{j=1}^l \Tr(H^{\vartheta_j}\cdot A_j)\right\},$$
where $A_j\in{ \R[\pi_1,\ldots,\pi_n]}^{\eta_j\times\eta_j}$ is a sum of squares matrix polynomial.
\end{theorem}
Although the above Theorem \ref{thm:sums of squares description} is not too different in spirit from Theorem \ref{THM Decomp},  it allows to transfer the decision if an invariant polynomial is a sum of squares into the invariant ring. Furthermore, this in turn allows for a more global approach. Indeed, Theorem \ref{THM Decomp} was stated for polynomials of a given degree $2d$ and relied in the decomposition of the space $\R[X_1,\ldots,X_n]_{\leq d}$. The use of the covariant algebra implicitly also directly yields a decomposition of the space $\R[X_1,\ldots,X_n]_{\leq d}$. Therefore, understanding the decomposition of the covariant algebra also provides important quantitative information. More precisely:
we denote  the dimension of the $\R$-vector space of $G$-invariant forms of degree $d$ by  $N_G(d)$, and write $h_k^{\vartheta}$ for the multiplicity of an irreducible representation $\theta$ in $(\R[X_1,\ldots,X_n]_G^{\theta})_{\leq k}$, i.e., for  the multiplicity of $\theta$ in the isotypic decomposition of the subspace of polynomials  of degree at most $k$ in the covariant algebra.
Then we get the following quantitative information necessary for applying Theorem \ref{THM Decomp}  directly. 

\begin{proposition}\label{cor:bound}
Let $G$ be a finite reflection group and $\theta$ be an irreducible representation. Then the multiplicity of the corresponding irreducible representation in the $G$-module $(\R[X_1,\ldots,X_n]_G^{\theta})_{\leq k}$ equals  
$$\sum_{k=0}^dN_G(d-k)\cdot h^{\vartheta}_k.$$
\end{proposition}
We are going to explore this theorem and its consequences a bit more in detail in the case of symmetric polynomials in the next Subsection and in the situation of a regular triangle in the next example,  which is made more general in  \cite{debus2020reflection}.
\begin{example}{Sums of squares invariant by the dihedral group $D_3$}
Consider the dihedral group $D_{3}$ of order $6$ acting on the plane as a reflection group. The corresponding invariant ring is generated by the two  polynomials  $\pi_1=X_1^2+X_2^2$ and $\pi_2= X_2^3 - 3X_1X_2^2$. The group  has three irreducible representations, two of which are one dimensional and one of which is two dimensional.  We find that the corresponding covariant algebra $\R[x,y]_{D_{3}}$  decomposes into $$\theta^{(1)} = \langle 1 \rangle, \quad \theta^{(2)} = \langle -X_1^3+3X_1X_2^2 \rangle, \quad \theta^{(3)}_1 = \langle X_1,X_2\rangle, \quad \theta^{(3)}_2 = \langle X_1X_2,X_2^2 \rangle.$$ Then $$H^{\theta^{(1)}} = (1),\; H^{\theta^{(2)}} = \left( \mathcal{R}_{D_3} (3X_1X_2^2-X_1^3)^2\right), H^{\theta^{(3)}} = \left(\begin{array}{cc}
 \mathcal{R}_{D_3} (X_1^2)    & \mathcal{R}_{D_3}(X_1^2X_2) \\
  \mathcal{R}_{D_3} (X_1^2X_2)   & \mathcal{R}_{D_3} (X_1^2X_2^2) 
\end{array}\right), $$
and we obtain
$$H^{\theta^{(1)}} = (1),\; H^{\theta^{(2)}} = \left( \pi_1^3-\pi_2^2\right),\; H^{\theta^{(3)}} = \left(\begin{array}{cc}
 \phantom{-}\frac{\pi_1}{2}    & -\frac{\pi_2^2}{2} \\
  -\frac{\pi_2^2}{2}   & \phantom{-}\frac{\pi_1^2}{8}  
\end{array}\right), $$
\end{example}

\subsection{Symmetric sums of squares}\label{ssec:Symsos}
To conclude our discussion on the representations of invariant sums of squares, we focus on the case of symmetric polynomials. This case has been studied by various authors (for example \cite{kurpisz2020sum,riener2011symmetries,raymond2018symmetric,heaton2021sos}), and our presentation follows \cite{riener2}. 

In Definition \ref{def:Specht} we have seen how to combinatorially describe the irreducible representations of $\S_n$ with the help of Young tabloids. A classical construction of Specht realizes these irreducible representations  as submodules of the polynomial ring (see \cite{specht1937darstellungstheorieSn}): 
\begin{definition}
For $\lambda\vdash n$ let $T_{\lambda}$ be a  standard Young tableau of shape $\lambda$ and
$\mathcal{C}_1,\ldots,\mathcal{C}_{\nu}$ be the columns of
$T_\lambda$. To $T_\lambda$ we associate the monomial
\[X^{T_{\lambda}}:=\prod_{i=1}^{n}X_i^{m(i)-1},\] where $m(i)$ is the
index of the row of $T_{\lambda}$ containing $i$.
Note that for any $\lambda$-tabloid $\{T_{\lambda}\}$ the monomial
$X^{T_{\lambda}}$ is well defined,
and the mapping $\{T_{\lambda}\} \mapsto X^{T_{\lambda}}$ is a
$\mathfrak{S}_n$-isomorphism.
For any column $\mathcal{C}_i$ of $T_\lambda$ we denote by
$\mathcal{C}_i(j)$ the  element in the  $j$-th row and we associate to it a
Vandermonde determinant:
$$\Van_{\mathcal{C}_{i}} \ := \ \det
\left(
\begin{array}{ccc}
X_{ \mathcal{C}_i(1)}^0& \ldots  &X_{\mathcal{C}_i(k)}^0   \\
 \vdots&  \ddots &\vdots   \\
X_{ \mathcal{C}_i(1)}^{k-1}& \ldots  &X_{\mathcal{C}_i(k)}^{k-1}
\end{array}
\right) \ = \ \prod_{j<l}(X_{\mathcal{C}_i(j)}-X_{\mathcal{C}_i(l)}).$$
\noindent The \emph{Specht polynomial} $sp_{T_{\lambda}}$ associated
to $T_\lambda$ is defined as
\[
  sp_{T_{{\lambda}}} \ := \ \prod_{i=1}^{\nu} \Van_{\mathcal{C}_{i}}
  \ = \
 \sum_{\sigma\in \CStab_{T_{\lambda}}}\sgn(\sigma)\sigma(X^{T_{\lambda}}) \, ,
\]
where $\CStab_{T_{\lambda}}$ is the column stabilizer of $T_\lambda$.
\end{definition}
\noindent By the $\mathfrak{S}_n$-isomorphism $\{T_{\lambda}\} \mapsto
X^{T_{\lambda}}$,
$\mathfrak{S}_n$ acts on $sp_{T_{{\lambda}}}$ in the same way as on
the polytabloid $e_{T_{\lambda}}$.
If $T_{\lambda,1},\ldots,T_{\lambda,k}$ denote all standard Young
tableaux associated to $\lambda$, then the set of polynomials
$sp_{T_{\lambda,1}},\ldots,sp_{T_{\lambda,k}}$
are called the \emph{Specht polynomials} associated to $\lambda$. We then have the following result due to Specht \cite{specht1937darstellungstheorieSn}:
\begin{proposition}\label{pro:Specht}
The Specht polynomials $sp_{T_{\lambda,1}},\ldots,sp_{T_{\lambda,k}}$ span an
$\mathfrak{S}_n$-submodule of $\R[X]$ which is isomorphic to the Specht
module $S^{\lambda}$.
\end{proposition}
Using the construction of Specht polynomials we aim now to obtain a version of Theorem \ref{thm:sums of squares description}  in the case of the symmetric polynomials. For this we rely on a basis for the corresponding covariant ring. Such a basis was described by Ariki, Terasoma, and Yamada  in \cite{ATY}  with the construction of the so-called   \emph{Higher Specht polynomials}.  These polynomials generalize the Specht polynomials in such a way that they yield a concrete basis as in Proposition \ref{cor:deomposition of polynomial ring}. Their definition is given by a pair of Young tableaux. 

\begin{definition}
Let $n\in\N$ be a natural number.
\begin{enumerate}
\item We call a finite sequence $w=(w_1,\ldots,w_n)$ of non-negative integers  a \emph{word of length $n$}. 
\item A word of length $n$ is called a \emph{permutation} if the set of its  non-negative integers is $\{1,\ldots,n\}$.
\item Given a word $w$ and a permutation $u$ we define the monomial associated to the pair as  $X_u^{w}:=X_{u_1}^{w_1}\cdots X_{u_n}^{w_n}$.
\item To a given permutation $w$ its index  denoted by $i(w)$, is the word of length $n$ constructed the following way:  the word $i(w)$ contains 0 exactly at the same position where $1$ occurs in $w$. The other entries are defined recursively with the following rule: if $k$ is in position $c$ in the word $w$ (that is, $w_c=k$) and $k+1$ is in position $d$, then $i(w)_d=i(w)_c$ if $d>c$, and $i(w)_d=i(w)_c+1$ otherwise. 
\item For a partition $\lambda$ of $n$ and a standard Young tableau of shape $\lambda$, $T$ the \emph{word} of $T$ - denoted by $w(T)$ - is defined by collecting the entries of $T$ from the bottom to the top in consecutive columns starting from the left. 
\item For a pair  $(T,V)$ of standard $\lambda$-tableaux we define the monomial associated to this pair as $X_{w(V)}^{i(w(T))}.$ The degree of the monomial is called the \emph{charge of $T$}, denoted $c(T)$.
\end{enumerate}
\end{definition}

\begin{example}{An example of a monomial built from a tableau}
Consider the tableau $$\Yvcentermath1 T=\young(124,35).$$ The resulting word is given by $$w(T)=31524,$$ 
with  $$i(w(T))=10201.$$ 
Taking $$\Yvcentermath1 V= \young(135,24)$$ we obtain the monomial  $X_{w(V)}^{i(w(T))}=X_2^1X_1^0X_4^2X_3^0X_5^1 = X_2X_4^2X_5$ of degree $c(T)=1+0+2+0+1=4$.
\end{example} 

\begin{definition}
Let $\lambda\vdash n$ and $T$ be a $\lambda$-tableau. Then the \emph{Young symmetrizer} associated to $T$ is the element in the group algebra $\R[\S_n]$ defined to be 
$$\varepsilon_T=\sum_{\sigma\in \RStab_T}\sum_{\tau\in \CStab_T} \sgn( \tau) \tau\sigma.$$
Now let $T$ be a standard Young tableau, and define the \emph{higher Specht polynomial} associated with the pair $(T,V)$ to be 
$$F_V^T(X_1,\ldots,X_n):=\varepsilon_V(X_{w(V)}^{i(w(T))}).$$

\end{definition}
The importance of these higher Specht polynomials now is summarized in the following Theorem which can be found in  \cite[Theorem 1]{ATY}.
\begin{theorem}\label{thm:higher}
The following holds for the set of higher Specht polynomials.
\begin{enumerate}
\item The set $\mathcal{F}=\bigcup_{\lambda\vdash n} \mathcal{F}_\lambda$ is a basis of the covariant ring $\R[X]_{\S_n}$ over $\R[X]^{\S_n}$.
\item For any $\lambda\vdash n$ and standard $\lambda$-tableau $T$, the space spanned by the polynomials in $$\mathcal{F}^T_\lambda:=\{F^T_V, \text{ where } V \text{ runs over all standard } \lambda\text{-tableaux} \}$$ is an irreducible $\S_n$-module isomorphic to the Specht module $W^{\lambda}$. 
\end{enumerate}
\end{theorem}
\begin{example}{The higher Specht polynomials for $\S_3$}
Consider the group $\S_3$. The complete list of higher Specht polynomials is given by
\begin{align*}
  \{1\}& \text{ for }W^{(3)}&\\
  \left\{(X_2-X_1),\,(X_3-X_1)\right\},\;\left\{X_3(X_2-X_1),\,X_2(X_3-X_1)\right\}&\text{ for }W^{(2,1)}&\\
  \left\{(X_2-X_3)(X_3-X_1)(X_3-X_2)\right\}&\text{ for }W^{(1,1,1)}.&
\end{align*}
These correspond to the trivial representation $W^{(3)}$, the 2-dimensional Specht module $W^{(2,1)}$ and the 1-dimensional  Specht module $W^{(1,1,1)}$.
We thus compute
\begin{align*}
H_{(3)}&=1\\\\
H_{(2,1)}&=\begin{pmatrix}
R_{\S_3}((X_2-X_1)^2)&R_{\S_3}((X_2-X_1)X_3(X_2-X_1))\\
R_{\S_3}((X_2-X_1)X_3(X_2-X_1))&R_{\S_3}(X_3(X_2-X_1)^2)
\end{pmatrix}\\
&=\begin{pmatrix}
\pi_2-\frac{1}{3}\pi_1^2&-\frac{1}{3}\pi_1^3+\frac{4}{3}\pi_1\pi_2-\pi_3\\
-\frac{1}{3}\pi_1^3+\frac{4}{3}\pi_1\pi_2-\pi_3\quad &
-\frac{1}{6}\pi_1^4+\frac{2}{3}\pi_1^2\pi_2-\frac{2}{3}\pi_1\pi_3+\frac{1}{6}\pi_2^2
\end{pmatrix}\\\\
H_{(1,1,1)}&=R_{\S_3}((X_1-X_2)(X_1-X_3)(X_2-X_3))\\
&=\frac{1}{6}(-{{\pi_1}}^{6}+9\,{{\pi_1}}^{4}{\pi_2}-8\,{{\pi_1}}^{3}{\pi_3}-21
\,{{\pi_1}}^{2}{{\pi_2}}^{2}+36\,{\pi_1}\,{\pi_2}\,{\pi_3}+3\,{{
\pi_2}}^{3}-18\,{{\pi_3}}^{2})
\end{align*}
Using Theorem \ref{thm:sums of squares description} one now obtains  that every symmetric sum of squares polynomial $f\in\R[X_1,X_2,X_3]^{\S_3}$  can be written in the from
\[
f=\sigma_0(\pi_1,\pi_2,\pi_3)+\sigma_1(\pi_1,\pi_2,\pi_3) H_{(1,1,1)}+Tr(M(\pi_1,\pi_2,\pi_3)\cdot H_{(2,1)}),
\]
where $\sigma_0(\pi_1,\pi_2,\pi_3)$, and $\sigma_1(\pi_1,\pi_2,\pi_3)$ are sums of squares in $\R[\pi_1,\ldots,\pi_3]$ and $M\in\R[\pi_1,\ldots,\pi_3]^{2\times 2}$  is a sum of squares matrix polynomial.
\end{example}
Furthermore, we can use the combinatorial description of the Higher Specht polynomials to gain understanding of the ``complexity'' of  symmetric sums of squares decomposition of polynomials of fixed degree $2d$. Indeed, we find that the  multiplicities $m_{\lambda}$ of the $\S_n$-modules $W^\lambda$ which appear in an isotypic decomposition  $\R[X_1,\ldots,X_n]_{=d}$ coincide with  the number of standard $\lambda$-tableaux $S$ with charge  at most $d$, i.e., all $S$ with $c(S)\leq d$.
This yields in particular the following observation which was first remarked in \cite[Theorem 4.7.]{riener2013exploiting}, and can be generalized to all infinite families of reflection groups \cite[Theorem 3.31]{debus2020reflection}.
\begin{theorem}
\label{thm:stabel}
Let $H_{n,2d}^{\S_n}$ denote the homogeneous symmetric polynomials of degree $2d$. Then the matrix size needed to decide if $f\in H_{n,2d}^{\S_n}$ is a sum of squares do not depend on $n$ once $n\geq 2d$.
\end{theorem}
A similar phenomenon was also shown in  situations where $\S_n$ is not acting as a reflection group \cite{raymond2018symmetric}.
The stabilization observed in Theorem \ref{thm:stabel} makes  it  in particular possible to give uniform descriptions for symmetric sums of squares of a given degree. For example we have the following representation theorem of symmetric quartics which was derived in \cite{blekherman2012nonnegative} with the methods presented in this section.

\begin{theorem}\label{thm:decom}
Let $f\in H_{n,4}^{\S_n}$ be a symmetric quartic. Then  $f$
is a sum of squares if and only if it can be written in the form
\begin{eqnarray*}
f^{(n)}&=&\alpha_{11}\pi_1^4+2\alpha_{12}\pi_1^2\pi_2+\alpha_{22}\pi_2^2\\
&+&\beta_{11}\left(\pi_1^2\pi_2-\pi_1^4\right)+2\beta_{12}\left(p_{(3,1)}-\pi_1^2\pi_2\right)+\beta_{22}\left(\pi_4-\pi_2^2\right)\\
&+&\gamma\left(\frac{1}{2}
\pi_1^4-\pi_1^2\pi_2+\frac{n^2-3n+3}{2n^2}\pi_2^2+\frac{2n-2}{n^2}
\pi_1\pi_3+\frac{1-n}{2n^2} \pi_4\right),
\end{eqnarray*}
where  $\pi_j=\frac{1}{n}(X_1^j+\ldots X_n^j)$ for $1\leq j\leq 4$  and the parameters $\alpha_{11},\alpha_{12}\,\alpha_{22},\beta_{11},\beta_{12},\beta_{22}$ are chosen 
such that $\gamma\geq 0$ and the matrices $\begin{pmatrix}
\alpha_{11}&\alpha_{12}\\
\alpha_{12}&\alpha_{22}
\end{pmatrix}
$ and $\begin{pmatrix}
\beta_{11}&\beta_{12}\\
\beta_{12}&\beta_{22}
\end{pmatrix}$ are positive semidefinite.
\end{theorem}

\section{Miscellaneous approaches}\label{ssec:Misc}
So far we have been focused on ways to explore symmetries in sums of squares decompositions which are used to obtain an approximation hierarchy for  polynomial optimization problems. To close our discussion, we present here some other ways to use symmetries in optimization problems. The selection we present here is non exhaustive and the presentation of the individual methods is rather short. However, we hope this overview will provide useful points of references for the reader. 

\subsection{Orbit spaces and polynomial optimization}\label{ssec:OrbSp}
The methods of invariant theory presented in Section \ref{sec:invariant} have provided that we can represent an invariant polynomial algebraically in terms of generators of the invariant ring. Indeed,  Theorem \ref{thm:hilbertfinite} yields that  $\K[X]^G$ is finitely generated and therefore is the coordinate ring of an algebraic variety. So we can associate to a choice of generators $\pi_1,\ldots, \pi_m$ and  the corresponding  inclusion $\K[\pi_1,\ldots,\pi_m]\subseteq \K[X_1,\ldots,X_n]$ a morphism $\Pi$ defined explicitly via
\[\
\abb{\Pi}{\K^n}{\K^m}{x}{(\pi_1(a),\ldots,\pi_m(a))}.\] This map is also called the \emph{Hilbert map}.
Let $a\in\K^n$, then the \emph{orbit of $a$} denoted by $G_a$ is the set of points to which $a$ is mapped to under the action of $G$, i.e.
$$G_a:=\{g(a)\:\ g\in G\}.$$ 
As the orbits $G_a$ and $G_b$ of two points in $\K^n$ are either equal or disjoint, 
the action of $G$ on $\K^n$ naturally defines an equivalence relation by $a\sim b$ if and only if $b=g(a)$ for some $g\in G$ and we can consider the set of  equivalence classes, i.e.,  the set of all $G$-orbits on $\K^n$ and denote this  by $\K^n/G$. This set is called the \emph{orbit space}. Notice that the Hilbert map is constant on $G$-orbits, so it makes sense to view it as a mapping of the orbit space.
Now if have a finite group $G$ acting on a linear space defined over an  algebraically closed field, for example $\K=\C$, then  
the polynomial mapping defined above is surjective onto the $n$ dimensional variety $V_\Pi\subset\C^m$ defined by the algebraic relations between the $m$ polynomials $\pi_1,\ldots,\pi_m$. In particular, if the polynomials are algebraically independent, which is exactly the case if $G$ is a finite reflection group, each orbit is mapped to a point in $\C^n$. Moreover if $g_1,\ldots,g_k$ are invariant polynomials that describe an algebraic set in $\C^n$ the Hilbert map sends this set to a new algebraic set in $\C^n$ which is given by the polynomials $\gamma_1,\ldots,\gamma_k$ which satisfy $\gamma_j(\pi_1,\ldots,\pi_m)=g_j\,\text{ for } 1\leq j\leq k$.
Therefore, the possibility to represent invariant functions in terms of generators can reduce the complexity  of the description of invariant algebraic sets. 

In contrast to the algebraically closed case where the Hilbert map is surjective, if we restrict $\Pi$ to $\R^n$  the resulting map $\tilde{\Pi}$ may fail to be surjective. 
\begin{example}{Failure of surjectivity in the $D_4$ case}
Let $G=D_4$ be the Dihedral group acting on $\R^2$. Then the invariant ring $\C[X,Y]^{D_4}$ is generated by  $\pi_1=X^2+Y^2$ and $\pi_2=X^2Y^2$ and since $D_4$ is a reflection group $\pi_1$ and $\pi_2$ are  algebraically independent. The Hilbert map thus provides a connection between the orbit space  $\C^{2}/D_4$ and $\C^2$. However, if we restrict the map $\pi$ to $\R^2$ the  image of $\Pi$ is strictly contained in $\R^2$. Indeed, we must have  $\pi_1(x,y)\geq 0$ for all $(x,y)\in \R^2$, and thus for example $\Pi^{-1}(-1,0)\not\in\R^2$. Hence, the restricted map  $\tilde{\Pi}$ is not surjective. 
\end{example} 
Nevertheless, the real Hilbert map 
\begin{eqnarray*}
  {\Pi} : \R^n & \to & \R^n/G \ \subseteq \ \R^m \\
           a & \mapsto & (\pi_1(a) , \ldots, \pi_m(a))
\end{eqnarray*}
defines an embedding of the orbit space into $\R^m$ and by the Traski-Seidenberg principle of real algebraic geometry its image is a semi-algebraic subset of $\R^m$. It can be shown that the failure of the map being surjective is related to the existence of abelian $2$-subgroups of $G$ (see \cite{broecker}). Therefore, as soon as the order of the group is even, additional semi-algebraic conditions need to be imposed. For example, in the example above, every point in the image needs to  satisfy  \[\Pi(\R^2)\subseteq \{(z_1,z_2)\in\R^2\ :\ z_1\geq 0\}\] since $\pi_1(x,y)\geq 0$ for all $(x,y)\in\R^2$.

It was shown by Procesi and Schwarz \cite{procesi1985inequalities} that the image of $\Pi$ is in fact  a basic semialgebraic set, i.e., there exist finitely many polynomial inequalities which are satisfied if and only if a point is in the image of the map. Moreover, in the case of compact groups, these inequalities can be obtained from the chosen fundamental invariants in a direct way:

For a polynomial $p$ we consider the differential $dp$ defined by
$dp = \sum_{j=1}^n \frac{\partial p}{\partial x_j} d x_j$.
For finite (compact) $G$ we have a $G$-invariant inner product
$\langle \cdot , \cdot \rangle$ which when carrying over the to the differentials yields
$\langle dp, dq \rangle = \sum_{j=1}^n \frac{\partial p}{\partial x_j} \cdot
  \frac{\partial q}{\partial x_j}$ . Since differentials of $G$-invariant polynomials are $G$-equivariant,  the inner products $\langle d\pi_i, d \pi_j \rangle$
$(i,j \in \{1, \ldots, m\})$ are $G$-invariant, and hence every
entry of the symmetric matrix polynomial
\[
  J \ = \ ( \langle d\pi_i, d \pi_j\rangle)_{1 \le i,j \le m}
\]
is $G$-invariant. With this construction Procesi and Schwarz \cite{procesi1985inequalities} have shown the following.
\begin{theorem} \label{pr:procesischwarz}
Let $G \subseteq \mathrm{GL}_n(\R)$ be a compact matrix group, and let
$\Pi = (\pi_1, \ldots, \pi_m)$ be fundamental invariants of $G$.
Then the orbit space is given by polynomial inequalities,
\[
  \R^n / G \ = \ \Pi(\R^n) \ = \ \left\{z \in \R^n \, : \, J(z) \succeq 0 , \,
     z \in V(I) \right\} \, ,
\]
where $I \subseteq \R[z_1, \ldots, z_m]$ is the ideal of relations of
$\pi_1, \ldots, \pi_m$.
\end{theorem}
This statement now allows to use invariant theory in order to rewrite an invariant optimization problem: If we consider a polynomial optimization problem of the form 
\begin{equation}\label{eq:pop1}
\inf \{ p(x) \,\text{s.t. }{g_1}(x)\geq 0,\ldots,{g_m}(x)\geq 0 \},
\end{equation}
where we assume that the polynomials $p,g_1,\ldots, g_m$ are invariant by a group $G$, then by choosing $\pi_1,\ldots,\pi_m$ which generate the invariant ring we obtain corresponding polynomials $\tilde{p},\tilde{g}_1,\ldots\tilde{g}_m\in\R[\pi_1,\ldots,\pi_m]$, which might be of smaller degrees than the original polynomials. By using  Theorem \ref{pr:procesischwarz} we may now translate the optimization problem \eqref{eq:pop1} to the following equivalent optimization problem:
\begin{equation}\label{eq:pop2}
 \inf \{ \tilde{p}(z)\, \text{s.t.\, } z \in V(I_\Pi),\tilde{g_1}(z)\geq 0,\ldots,\tilde{g_m}(z)\geq 0,\,  J(z) \succeq 0 \, \}.
\end{equation}
So the Hilbert map allows us to reformulate the initial polynomial optimization problem in new polynomial functions, which potentially can be of smaller degree or involve less variables, at the price that we obtain a new set of constraints that come from the fact that the Hilbert map fails in general to be surjective.
\begin{example}{The choice of generators matters}
The matrix polynomial $J$, as well as the specific description of the original problem in terms of generators, depend on the choices made for the generators. Thus different choices might lead to quite different optimization problems. 
Consider the Motzkin polynomial $M=X_1^4X_2^2+X_1^2X_2^4-3X_1^2X_2^2+1$. We can rewrite this polynomial in terms of elementary symmetric polynomials, i.e.,
$\pi_1=X_1+X_2$ and $\pi_2=X_1X_2$. 
Then 
\[\tilde{M}_e=\pi_1^2\pi_2^2 - 2
\pi_2^3 - 3\pi_2^2 + 1.\]
The corresponding matrix polynomial $J_e$ is given by  \[J_e=\begin{pmatrix}2&\pi_1\\\pi_1&\pi_1^2-2\pi_2\end{pmatrix}.\]
 Now the minimal value of $M$ on $\R^2$ is the same as the minimal value of $\tilde{M}_e$ in the set of points $\left\{(\pi_1,\pi_2)\in \R^2\,:\,J_e(\pi_1,\pi_2)\succeq 0\right\}$.
On the other hand we can choose the power sum polynomials 
$\pi_1=X_1+X_2$, $\pi_2=X_1^2+X_2^2$ and we  obtain \[\tilde{M}_p=\frac{1}{4}(\pi_1^{4}\pi_2-3\,{\pi_1}^{4}-2\,{\pi_1}^{2}{\pi_2}^{2}+6
\,\pi_1^{2}\pi_2+\pi_2^{3}-3\,\pi_2^{2}+4).\]
Furthermore, we find for the corresponding matrix polynomial \[J=\begin{pmatrix}2&\pi_1\\\pi_1&\pi_2\end{pmatrix}.\] Now the minimal value of $M$ on $\R^2$ is the same as the minimal value of $\tilde{M}$ in the set of points $\left\{(\pi_1,\pi_2)\in \R^2\,:\,J(\pi_1,\pi_2)\succeq 0\right\}$.
\end{example}
The main advantages of the approach sketched above is that it may reduce the degrees of the polynomials describing the optimization problem and even can reduce the number of variables. Furthermore, optimal points to the equivalent optimization problem \eqref{eq:pop2} correspond to entire orbits of optimal points. Therefore the number of such points may also be drastically smaller. Finally, in contrast to the methods described in the main parts of this chapter, which were designed for sums of squares approximations, this approach reduces the symmetry on the formulation and therefore allows integration with other methods for solving polynomial optimization problems. However, it was observed in \cite{RTJL} that the special structure of the optimization problem \eqref{eq:pop2} - namely the fact that the additional constraints are expressed in terms of a \emph{polynomial matrix inequality}  $J(z)\succeq 0$ - allows to use a specially adapted version of the SOS-moment hierachy established by Hol and Scherer (and dually by Henrion and Lasserre). We briefly sketch this approach. 
\begin{trailer}{Orbit space formulations and moment relaxations}
For a polynomial matrix $G(x)$, define the localizing matrix as follows
$$M(G\star\,y)_{i,j,l,k}=L(b_i\cdot b_j\cdot G(x)_{l,k}) \,.$$
With this construction  one can define the following moment relaxation for the orbit space formulation of a polynomial optimization  problem defined in  \eqref{eq:pop2} for $k\in\N$:
\begin{equation}\label{qutient}
  Q_k^q:\quad\begin{array}{rcl}
   \multicolumn{3}{l}{\inf_{y} \sum_{\alpha} \tilde{p}_{\alpha} y_{\alpha}} \\
  M_k(y) & \succeq & 0 \, , \\
  M_{k - \lceil \deg\tilde{g}_j / 2 \rceil}(\tilde{g}_j \, y) & \succeq & 0,\\
   M_{k-m}(J \star y) & \succeq & 0 \,.\\
  \end{array}
\end{equation}
For $k$ large enough,  each solution to the sequence of semi-definite programs defined in \eqref{qutient} provides a lower bound for the problem and moreover, under additional conditions, the sequence of relaxations converges. (see \cite{riener2013exploiting} for details). 
This relaxation approach recently has also been adapted to multiplicative group actions (see \cite{hubert}).
\end{trailer}

\subsection{Reduction via orbit decomposition}\label{ssec:AMGM}
In the symmetry reduction techniques for sums of squares presented in this chapter, the first  essential step consisted in the observation that for convex sets one can reduce easily to orbits and the second step was grounded in consequences of  Schur's lemma.  Approaches using the second step  are however clearly  limited to those  situations that allow for an application of Schur's lemma. However, even in cases where sophisticated tools from representation theory will not directly apply, one can obtain substantial reductions by cleverly decomposing invariant problems according to orbits.
A natural idea consists for example in turning a large symmetric problem into smaller problems, in such a way that solving these smaller problems gives at least a solution by orbit, therefore giving a full understanding of the solutions of the initial problem. 
Such ideas can be useful, and sometimes refined, in several situations, including in optimization and computation (see for instance \cite{faugere2023computing, herr2013exploiting, Schrmann2013ExploitingSI}). 
We showcase one example of this simple, but sometimes effective, idea in the context of polynomial optimization: the orbit decomposition of SAGE signomials. 

In order to motivate this situation we recall that the main key for the approximations methods used in the previous sections comes from the observation that a polynomial which  can be written as a sum of squares is non-negative. Of course, one can replace this condition  with other certificates  of non-negativity certification. Another example of such certificates relies on SAGE functions (see \cite{chandrasekaran2016relative}). 
A signomial is a function of the form \[f(x) = \sum_{\alpha \in \A}c_\alpha e^{<\alpha,x>}\] where $\A \subset \R^n$ is a finite set. Under the change of variables $y_i = e^{x_i}$, this can be seen as a generalization of polynomial functions restricted to  the positive orthant. Suppose a signomial $f$ can be written in the from  \begin{equation}\label{AGE}f(x) = \sum_{\alpha \in \A}c_\alpha e^{<\alpha,x>} + d e^{<\beta,x>}\end{equation}   $c_\alpha >0$ for $\alpha \in \A$ and $\beta \in \relint(\A)$. Then,  it follows from  arithmetic-geometric mean inequality that  $f$    is non-negative on $\R^n$ if and only if there exists $\nu \in R^\A_+$ such that \[\sum_{\alpha \in \A}\nu_\alpha \alpha =\left(\sum_{\alpha \in \A} c_\alpha\right)\beta\; \text{ and }\sum_{\alpha \in \A} \nu_\alpha \ln \frac{\nu_\alpha}{e \cdot c_\alpha} \leqslant d.\] 
A function of this form is called an \emph{AGE function}, and a sum of AGE functions is called a \emph{SAGE function}. Clearly, a SAGE function is non-negative on $\R^n$, and  
similarly to the sum of squares formulation one can define a \emph{SAGE-approximation} of the polynomial optimization problem \eqref{pop} via
 \[f^{SAGE} = \sup\{\lambda,\ f-\lambda \textrm{ is SAGE}\} \leqslant  \inf\{f(x),\ x \in \R^n\}
 \] 
to obtain a bound on the minimum of a given signomal function $f$ on $\R^n$. The observation which makes this formulation computationally viable is that the task to decide if a given signomial is SAGE can be decided via \emph{relative entropy programming} (a subclass of convex programs, see \cite{chandrasekaran2016relative}).

 Now, if we suppose that a signomial $f$ as written in \eqref{AGE} is invariant by the action of a group $G$ (linearly on the exponents) and denote by $\hat{\B}$ be a set of orbit representatives for $\B$ by this action, then  applications of  the Reynolds operator reveal (see \cite{moustrou2022symmetry} for details) that  $f$ is SAGE if and only if for every $\beta \in \hat{\B}$, there exists an AGE signomial \[h_{\beta} = \sum_{\alpha \in \A}c_{\beta,\alpha} e^{<\alpha,x>} + d_\beta e^{<\beta,x>}\] such that \[f= \sum_{\beta \in \hat{\B}} \sum_{\rho \in G/\Stab(\beta)} \rho h_\beta,\]
 where the functions $h_{{\beta}}$ can be chosen to be invariant under the action of the stabilizer $\Stab_G(\beta)$ of $\beta$.

\begin{example}{SAGE decomposition of a symmetric signomial}\label{exa:SAGEdecomposition} Consider the symmetric signomial \[f=5 e^{6x} + 5e^{6y} + 5e^{6z}  -e^{2x+y+z} - e^{x+2y+z} - e^{x+y+2z} +6.\] With the notation above  we have $\hat{\B}= \{ (1,1,2)\}$. The stabilizer of $(1,1,2)$ is $\S_2 \times \S_1$ and $f$ is SAGE if and only if there exists an AGE polynomial of the form \[g= c_1 e^{6x} + c_2 e^{6y} + c_3e^{6z} + c_4 - e^{x+y+2z}\] such that \[f= g + (1\ 3)g + (2\ 3)g.\] Moreover, since we can assume that $g$ is invariant under the action of $\S_2 \times \S_1$, we can assume that $c_1=c_2$. Identifying the coefficients we find therefore that $c_1= c_2 = 1$, $c_3=3$ and $c_4=2$. Thus, $f$ is SAGE if and only if \[g= e^{6x} + e^{6y} + 3e^{6z}  - e^{x+y+2z} +1\] is AGE, and therefore we directly arrive at a substantial reduction of the associated relative entropy program. 
\end{example}
In fact the reduction of complexity suggested in the example above can be made precise by understanding the orbits and stabilizers of the exponents appearing in the signomial.
In particular in the case of symmetric AGE functions, i.e., invariant by the group $\S_n$, it can be shown (see \cite[Theorem 5.2]{moustrou2022symmetry}) that one can expect a stabilization of complexity with $n$ similar to Theorem \ref{thm:stabel}: for some sequences of signomials, if the number of terms grows quickly, the number of variables and constraints in the corresponding relative-entropy program remains constant for $n$ large enough. Therefore, even in a situation where we are not able to use the power of representation theory, we can turn large optimization problems involving a lot of variable and constraints into a smaller one that can be solved efficiently.

\subsection{Symmetries of optimizers}\label{ssec:SymSol}
We conclude this section with a rather general approach, which may be used in some situations. Again, let $G$ be a group acting linearly on a vector space $V$. Given an optimization problem that is symmetric, i.e., invariant by a group $G$, one could expect that also the solutions exhibit symmetries. Firstly, and clearly, the set of optimal points will be closed under the action of the symmetry group. In the best case, also the optimal points themselves will be fixed by the group.  If we know apriori that this is the case, we only need to look at the linear subspace of points which are invariant by $G$. If the action of $G$ is not the trivial one, this subspace will be of smaller dimension thus, one obtains directly a reduction of the optimization problem. The following example is one easy and simple prototype of this idea.

\begin{example}{A symmetric problem with a symmetric solution}
Given $a>0$ we want to find  the maximal area of an rectangle of perimeter $a$.  The solution is the square and one find directly that the side length is $\frac{1}{4}a$.
\end{example}
However,  it is in general not clear what conditions on the optimization problem guarantee this property (see for example \cite{waterhouse} for a beautiful exhibition of situations and some historical panorama). But in particular, if there is  just one optimal point, this optimal point has to be invariant by the group. So, we can conclude that in particular for convex optimization problems, this approach might directly lead to a reduction of dimension. Building on this simple idea, it may also be beneficial if one knows apriori that the optimal points are fixed by some (non trivial) subgroup of $G$.  For a subgroup $G'\subset G$ we denote by $V^{G'}$ the elements in $V$ fixed by $G'$.
It is a simple observation that $V^{G'}$ is a vector subspace of $V$ which might be of smaller dimension. Thus by restricting the original optimization problem to the smaller vector space $V^{G'}$ we reduce the dimension of the problem.

We present some situations in which this approach can lead to some substantial reduction. Building on the example above, suppose that we deal with an optimization problem defined in symmetric polynomials and we thus consider the action of the symmetric group $\S_n$ on $\R^n$. The action of $\S_n$ on $\R^n$ naturally decomposes the space into orbits. 
\begin{definition}
For every $x\in \R^n$, the associated stabilizer subgroup $\Stab(x)\subseteq \S_n$ is of the form
$$\Stab(x) \simeq \S_{\ell_1}\times \S_{\ell_2}\times \cdots \times \S_{\ell_k}$$
with $\ell_1\geq \ell_2\geq \ldots\geq \ell_k$. We hence define the \emph{orbit type} of $x$ to be  
\[
\Lambda(x):=(\ell_1, \ell_2,\cdots, \ell_k).
\]
Then, for a given $\lambda\vdash n$ we can define $$H_\lambda:=\left\{x\in \K^n\,:\, \Lambda(x)=\lambda\right\},$$
 and
$$A_k:=\{ x\in\R^n \,:\, \text{ the orbit type of } x \text{ has length at most k}\}$$
\end{definition}
\begin{remark}
For $k\in\N$, the set $A_k$ as defined above is a union of $k$-dimensional linear (sub)spaces of $\R^n$. 
\end{remark}
\begin{example}{Example $A_1$}
For $k=1$ we find $A_1=\{t\cdot(1,\ldots,1),\, t\in\R\}$ and in general $A_d$ can be identified with the set of points having at most $d$ distinct coordinates
\end{example}
Now we have the following (see \cite[Theorem ~4.5]{riener2012degree}, and \cite{timofte-2003}).

\begin{theorem}\label{thm:degree}
Let $f,g_1,\ldots,g_m \in \R[X]^{\S_n}$ and let 
set $$r:=\max\big\{2,\lfloor (\deg f)/2\rfloor, \deg g_1,\ldots,\deg g_m\big\}.$$ Consider the optimization problem 
$$f^*=\inf_{x\in K} f(x), \text{ where } 
K=\{x \in \R^n \, : \, g_1(x)\geq 0,\ldots, g_m(x)\geq 0\}.$$ 
If the set of optimizers is not empty it contains at least one point $x^*\in A_r$, and in fact 
$$f^*=\inf_{x\in K\cap A_r} f(x)$$
\end{theorem}

Now, a very direct approach to make use of this result consists in restricting to the different sub-spaces that build up $A_r$ individually. 
We consider the set $\Lambda_{r,n}$ of all partitions $\lambda$  of $n$ into exactlt $r$ parts.  Then for each
$\lambda \in \Lambda_{r,n}$, set $$f^{\lambda}:=f(\underbrace{T_1,\ldots,T_1}_{\lambda_1},\underbrace{T_2,\ldots,T_2}_{\lambda_2},\ldots,\underbrace{T_{r},\ldots,T_{r}}_{\lambda_{r}}) \in \R[T_1, \ldots, T_r] \, .
$$
Similarly, let
$K^{\lambda}:=\{t \in\R^r \, : \, g_1^{\lambda}(t) \geq0,\ldots,{g_m^{\lambda}(t) \geq0\}}$.
With this easy substitution  the original optimization problem in $n$ variables can be transformed into a set of new optimization problems that involve
only $r$ variables,
\begin{equation}\label{eq:partition}
\inf_{x\in K} f(x) \ = \
\min_{\omega\in\Omega}\inf_{t\in K^{\omega}} f^{\omega}(t) \, .
\end{equation}
\begin{remark}
Note that $|\Lambda_{r,n}|\leq \binom{n+r}{r}$ and therefore, for fixed $r$, the amount of additional problems in $r$ variables grows only polynomially in $n$.
\end{remark}
The main advantage of the simple approach based on Theorem \ref{thm:degree} consists in the possibility to use any  method to solve the resulting $r$-dimensional polynomial optimization problems. In particular, it has been observed in \cite{riener2013exploiting,debus2020reflection} that combining this approach with the sum of squares approach can lead to  qualitative improvements of the approximation. There have been several generalizations of Theorem \ref{thm:degree}: the paper \cite{moustrou2021symmetric} provides a finer criterion than only the degree of the polynomials involved to decide on the orbit type of solutions.   In \cite{riener2016symmetric, schabert} special  classes of symmetric polynomials are  defined, depending on special representation of the polynomials in the power sum basis or in the basis of elementary symmetric polynomials. For these classes an approach identical to the one described above can be derived from \cite[Theorem 2.6]{riener2016symmetric} or  \cite[Theorem 24]{schabert}. Furthermore, it can be shown that a similar statement holds in fact for all finite reflection groups \cite{acevedo2016test,friedl2018reflection} and even in the setup of the symmetric group not acting as a reflection group \cite{pct}. Finally, via Morse theory the fact that optimal points to special symmetric optimization problems tend to have some high symmetry also has impact on the topological complexity of the orbit space of $\S_n$-invariant semi-algebraic sets (see \cite{BS,BS2}) and a refined statement can be used to give complexity reduction in various algorithmical questions  related to the topology of symmetric semi-algebraic sets (see \cite{basu2017efficient,basu2022vandermonde}). Finally, one can observe that that after the restriction to $A_r$ in some cases the resulting system will still have some symmetries. This was for example used in \cite{faugere2023computing,vu} to obtain more efficient algorithmic results for computing critical  points.

\begin{acknowledgement}
This expository article originates from  a minicourse  ``Symmetries in algorithmic questions in real algebraic geometry'' held during  the virtual POEMA (Polynomial Optimization, Efficiency Through Moments and Algebra) Learning Weeks in July  2020.
  The authors would like to thank  the anonymous
  reviewers for helpful comments.
\end{acknowledgement}

\bibliographystyle{abbrv}
\bibliography{dummy.bib}

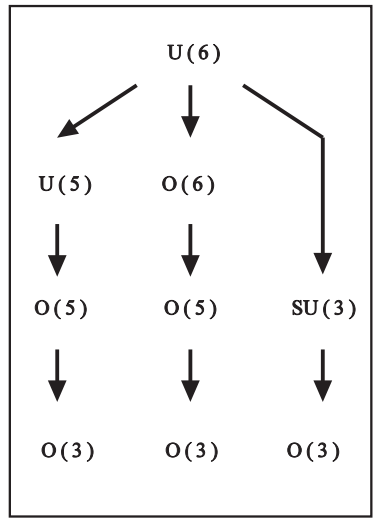
\end{document}